\documentclass{amsart}

\usepackage{amsmath, amssymb, amsthm, amsfonts, marginnote, stmaryrd}

\usepackage{color}

\input xy
\xyoption{all}

\usepackage{hyperref}

\swapnumbers
\numberwithin{equation}{section}

\theoremstyle{plain}
\newtheorem{theorem}[subsubsection]{Theorem}
\newtheorem{lemma}[subsubsection]{Lemma}
\newtheorem{prop}[subsubsection]{Proposition}
\newtheorem{cor}[subsubsection]{Corollary}
\newtheorem{conj}[subsubsection]{Conjecture}

\theoremstyle{definition}
\newtheorem{defn}[subsubsection]{Definition}

\newtheorem{remark}[subsubsection]{Remark}
\newtheorem{exam}[subsubsection]{Example}

\def\BB{\mathbb{B}}
\def\CC{\mathbb{C}}

\def\GG{\mathbb{G}}

\def\LL{\mathbb{L}}

\def\PP{\mathbb{P}}

\def\RR{\mathbb{R}}

\def\ZZ{\mathbb{Z}}

\newcommand\cB{\mathcal{B}}

\newcommand\cE{\mathcal{E}}
\newcommand\cF{\mathcal{F}}
\newcommand\cG{\mathcal{G}}
\newcommand\cH{\mathcal{H}}

\newcommand\cK{\mathcal{K}}
\newcommand\cL{\mathcal{L}}

\newcommand\cN{\mathcal{N}}
\newcommand\cO{\mathcal{O}}

\newcommand\cU{\mathcal{U}}
\newcommand\cV{\mathcal{V}}
\newcommand\cW{\mathcal{W}}

\def\bI{\mathbf{I}}

\def\bK{\mathbf{K}}

\def\bP{\mathbf{P}}

\newcommand\frM{\mathfrak{M}}
\newcommand\frN{\mathfrak{N}}

\newcommand\frS{\mathfrak{S}}

\newcommand\frX{\mathfrak{X}}

\newcommand\frZ{\mathfrak{Z}}

\newcommand\frb{\mathfrak{b}}
\newcommand\frc{\mathfrak{c}}

\newcommand\frg{\mathfrak{g}}

\newcommand\frh{\mathfrak{h}}

\newcommand\frl{\mathfrak{l}}

\newcommand\frn{\mathfrak{n}}
\newcommand\frp{\mathfrak{p}}
\newcommand\frq{\mathfrak{q}}

\newcommand\frt{\mathfrak{t}}

\newcommand\tilW{\widetilde{W}}

\newcommand\bimon{\mathit{bimon}}

\newcommand{\Bun}{\textup{Bun}}

\newcommand\ev{\textup{ev}}

\newcommand{\Gr}{\textup{Gr}}

\newcommand{\Hk}{\textup{Hk}}

\newcommand\id{\textup{id}}
\renewcommand{\Im}{\textup{Im}}

\newcommand\inv{\textup{inv}}

\newcommand\Lie{\textup{Lie}}
\newcommand\Loc{\textup{Loc}}

\newcommand\pr{\textup{pr}}

\newcommand{\rel}{\textup{rel}}

\newcommand{\Res}{\textup{Res}}

\newcommand\rk{\textup{rk}}

\newcommand\Spec{\textup{Spec }}
\newcommand\Spf{\textup{Spf }}

\newcommand\Sym{\textup{Sym}}

\newcommand{\univ}{\textup{univ}}

\newcommand\Aut{\textup{Aut}}

\newcommand\GL{\textup{GL}}

\renewcommand\sl{\mathfrak{sl}}

\newcommand{\Gm}{\GG_m}

\newcommand{\ad}{\textup{ad}}
\newcommand{\Ad}{\textup{Ad}}

\newcommand\xcoch{\mathbb{X}_*}

\renewcommand\a\alpha
\renewcommand\b\beta
\newcommand\g\gamma
\newcommand\G\Gamma
\renewcommand\d\delta
\newcommand\D\Delta

\renewcommand{\th}{\theta}

\newcommand{\ph}{\varphi}
\newcommand{\Sig}{\Sigma}
\newcommand{\s}{\sigma}
\renewcommand{\t}{\tau}

\newcommand{\y}{\eta}
\newcommand{\z}{\zeta}

\renewcommand{\l}{\lambda}
\renewcommand{\L}{\Lambda}
\newcommand{\om}{\omega}
\newcommand{\Om}{\Omega}
\newcommand{\io}{\iota}

\DeclareMathOperator{\Eis}{Eis}

\newcommand{\incl}{\hookrightarrow}
\newcommand{\isom}{\stackrel{\sim}{\to}}

\newcommand{\surj}{\twoheadrightarrow}

\newcommand{\twtimes}[1]{\times^{#1}}

\newcommand{\wt}[1]{\widetilde{#1}}
\newcommand{\wh}[1]{\widehat{#1}}
\newcommand\quash[1]{}
\newcommand\un{\underline}
\newcommand\ov{\overline}
\newcommand\bs{\backslash}
\newcommand\ot{\otimes}
\newcommand{\op}{\oplus}
\newcommand{\tl}[1]{[\![#1]\!]}
\newcommand{\lr}[1]{(\!(#1)\!)}

\newcommand{\cohog}[2]{\textup{H}^{#1}({#2})}     
\newcommand{\cohoc}[2]{\textup{H}_{c}^{#1}({#2})}     

\newcommand\upH{\textup{H}}

\newcommand{\oll}[1]{\overleftarrow{#1}}
\newcommand{\orr}[1]{\overrightarrow{#1}}

\newcommand\dG{G^{\vee}}

\newcommand\fl{f\ell}
\newcommand\pt{\textup{pt}}

\newcommand\vn{\varnothing}

\newcommand\xr{\xrightarrow}
\renewcommand\c{\circ}

\newcommand{\quot}{/\hspace{-.25em}/}
\newcommand{\Sph}{\mathit{sph}}

\newcommand{\Maps}{\textup{Maps}}

\newcommand{\sph}{\mathit{sph}}

\newcommand{\aff}{\mathit{aff}}

\newcommand{\Sh}{\mathit{Sh}}

\newcommand{\beq}{\begin{equation}}
\newcommand{\eeq}{\end{equation}}

\newcommand{\ssupp}{\mathit{ss}}

\newcommand\sss{\subsubsection}

\newcommand\IndCoh{\operatorname{IndCoh}}

\newcommand\na{\natural}
\renewcommand\r{\rho}

\newcommand\bu{\bullet}

\newcommand\lax{\textup{lax}}

\title{The global nilpotent cone for universal curves}
\author{David Nadler}
\address{Department of Mathematics\\University of California, Berkeley\\Berkeley, CA  94720-3840}
\email{nadler@math.berkeley.edu}
\author{Zhiwei Yun}
\address{Department of Mathematics, MIT,  77 Massachusetts Ave., Cambridge, MA 02139}
\email{zyun@mit.edu}

\subjclass[2020]{14D24}

\dedicatory{}	
\keywords{Hitchin fibration, global nilpotent cone, Betti geometric Langlands, singular support}

\begin{document}


\begin{abstract}
We construct a conic Lagrangian in the cotangent bundle of the moduli stack of $G$-bundles over the universal curve, restricting to the global nilpotent cone for each curve. It gives rise to a singular support condition suitable for the Betti geometric Langlands correspondence for families of curves and the automorphic gluing functor studied in \cite{NY-glue}. We also prove a family version of ``local constancy of Hecke operators," generalizing a result from \cite{NY}. 


\end{abstract}

\maketitle

\tableofcontents

\section{Introduction}

\subsection{Betti geometric Langlands}
Let $X$ be a smooth connected projective algebraic curve over $\CC$ and $G$ be a reductive group over $\CC$. In Betti geometric Langlands \cite{BN}, the automorphic category consists of sheaves on $\Bun_{G}(X)$ with singular support in the {\em global nilpotent cone} $\cN_X \subset T^*\Bun_G(X)$, i.e. the zero fiber of the Hitchin map. The Betti geometric Langlands equivalence, conjectured in \cite{BN} and proved in the recent breakthrough \cite{GLC}, asserts that the automorphic category $\Sh_{\cN_X}(\Bun_G(X))$ is equivalent to the spectral category $\IndCoh_{\cN^\vee}(\Loc_{\dG}(\un X))$, where $\Loc_{\dG}(\un X)$ is the (derived) moduli stack of Betti $\dG$-local systems on $X$, which only depends on the underlying topological surface $\un X$ of $X$.

Let $\pi: \frX\to S$ be a family of smooth projective curves. The spectral categories of the fibers $\frX_s=\pi^{-1}(s)$ vary locally constantly as $s$ varies in $S$, because the family of topological surfaces $\un{\frX_s}$ are locally trivializable. The Betti geometric Langlands equivalence then implies that the automorphic categories of $\frX_s$ also vary locally constantly. This is far from being clear from the definition of the automorphic categories which heavily depends on the complex structure of $\frX_s$.

It would be desirable to prove the local constancy of the automorphic categories for the family $\frX\to S$ directly without reference to the spectral side. In fact, it will take some effort even to formulate the local constancy of the automorphic categories precisely, which will be done in this paper.

\subsection{The universal nilpotent cone} 




Consider the family $\Pi: \Bun_{G}(\pi)\to S$ of moduli stacks of $G$-bundles along fibers of $\pi$. As $s$ moves in $S$, the automorphic categories $\Sh_{\cN_{\frX_{s}}}(\Bun_{G}(\frX_{s}))$ can be informally thought of as a sheaf of categories over $S$. A natural approach to establish local constancy is to construct a ``connection'' on this sheaf of categories, i.e., we need to specify which local sections of this sheaf are ``horizontal''. In particular, we need to specify which objects in the global category $\Sh(\Bun_{G}(\pi))$ are horizontal. In this paper, we propose a definition of such horizontal sheaves by imposing a singular support condition $\Bun_{G}(\pi)$. The key step in this definition is a construction of a conic Lagrangian in $T^*\Bun_G(\pi)$ that is a family version of the global nilpotent cone.

For any closed point $s\in S$, the fiber $(T^{*}\Bun_G(\pi))_{s}$ of $T^{*}\Bun_{G}(\pi)$ maps to the cotangent stack $T^{*}\Bun_{G}(\frX_{s})$
\begin{equation*}
p_{s}: (T^{*}\Bun_G(\pi))_{s}\to T^{*}\Bun_{G}(\frX_{s})
\end{equation*}
The first main result is:

\begin{theorem}[For full statement see Theorem \ref{th:Eis cone closed}] There exists a unique closed conic Lagrangian $\cN_{\pi}\subset T^*\Bun_G(\pi)$ such that for any closed point $s\in S$, $p_{s}$ restricted to the fiber $(\cN_{\pi})_{s}$ over $s$ gives a set-theoretic bijection $p_{s}: (\cN_{\pi})_{s}\to \cN_{\frX_{s}}$. 
\end{theorem}

In other words, $\cN_{\pi}$ is a lifting of the relative global nilpotent cone $\cN^{\rel}_{\pi}$ in the relative cotangent stack $T^{*}(\Pi)$. It is a general fact that there exists a unique closed conic Lagrangian in the absolute cotangent stack $T^{*}\Bun_{G}(\pi)$ that maps birationally to $\cN^{\rel}_{\pi}$ (see Proposition \ref{p:lifting Lag}). The fact that $\cN_{\pi}$ maps bijectively to $\cN^{\rel}_{\pi}$ is specific to our situation, whose proof relies on an explicit construction of $\cN_{\pi}$ called the {\em Eisenstein cone}: it is the transport of the zero section in $T^{*}\Bun_{B}(\pi)$ along the Lagrangian correspondence for cotangent bundles induced by the map $\Bun_{B}(\pi)\to \Bun_{G}(\s)$. For a single curve $X$, the observation that $\cN_{X}$ can be identified with the Eisenstein cone goes back to Ginzburg \cite{Gin}.


\subsection{Local constancy of automorphic category}




The local constancy of the automorphic categories can now be rigorous formulated as:

\begin{conj}\label{c:intro lc} Let $U\subset S^{an}$ be a contractible analytic open subset and $s\in U$. Then restriction along the inclusion $i_{s}: \{s\}\incl U$ induces an equivalence
\begin{equation*}
\Sh_{\cN_{\pi}}(\Bun_{G}(\pi_{U}))\isom \Sh_{\cN_{\frX_{s}}}(\Bun_{G}(\frX_{s})).
\end{equation*}
\end{conj}
In particular, there is an action of $\pi_{1}(S,s)$ on the automorphic category $\Sh_{\cN_{\frX_{s}}}(\Bun_{G}(\frX_{s}))$. Now starting with a fixed curve $X$ of genus $g$, considered as a fiber in the universal family of genus $g$ curves $\pi^{\univ}_{g}: \frX_{g}\to \frM_{g}$, where $X$ corresponds to a point $\xi\in \frM_{g}$. Then Conjecture \eqref{c:intro lc} implies there is a canonical action of the mapping class group $\pi_{1}(\frM_{g}, \xi)$ on the automorphic category $\Sh_{\cN_{X}}(\Bun_{G}(X))$. Of course, this action is expected to match the obvious mapping class group action on the spectral side.

Joakim F\ae rgeman and Marius Kj\ae rsgaard recently announced a proof of Conjecture \ref{c:intro lc} that does not rely on the Betti geometric Langlands equivalence. It however uses deep machinery developed in the works leading to \cite{GLC}. Therefore a more direct topological proof would still be of interest.

\subsection{Hecke stability of nilpotent sheaves}

A main result from \cite{NY} is that Hecke operators preserve the nilpotent singular support condition for automorphic sheaves in a strong sense. In particular, we show there that starting with $\cF\in \Sh_{\cN_{X}}(\Bun_{G})$ and applying the Hecke operator given by a kernel sheaf in the Satake category but at a varying point on the curve $X$, the singular support of the resulting sheaf on $X\times \Bun_{G}(X)$ is contained in the product $0_{X}\times \cN_{X}$, i.e., the result is locally constant along $X$.

In this paper we prove a family version of the above Hecke stability result. In the setup of a family of curve $\pi: \frX\to S$, let $S'$ be another stack smooth over $k$ with a map $t: S'\to S$ that lifts to $\frX$ 
\begin{equation*}
\xymatrix{ & \frX\ar[d]^{\pi}\\
S'\ar[r]^{t}\ar[ur]^{\xi} & S}
\end{equation*}
Let $\pi':\frX':=S'\times_{S}\frX\to S'$ be the base-changed curve.

\begin{theorem}\label{th:intro Hk} For any $\cK$ in the Satake category $\cH^{\Sph}$, the Hecke functor $H^{\Sph}_{\xi,\cK}$ sends $\Sh_{\cN_{\pi}}(\Bun_{G}(\pi))$ to $\Sh_{\cN_{\pi'}}(\Bun_{G}(\pi'))$.
\end{theorem}
For the precise definition of the Hecke functor $H^{\Sph}_{\xi, \cK}$ we refer to \S\ref{ss:Hk}. 

Theorem \ref{th:intro Hk} recovers \cite[Theorem 1.2.1]{NY} by taking $S=\pt$, $\frX=X$ the curve in question, and $S'=X$ with $\xi=\id$. We emphasize however that the proof we give here is very different from the one in {\em loc.cit}: there our proof is of computational nature while here it is a ``pure-thought'' proof by putting our statement in the appropriate generality. 

Both main theorems above are proved more generally for the version of $\Bun_{G}$ with either Iwahori level structure ($B$-reduction) or the unipotent-Iwahori level structure ($N$-reduction) along a multi-section of the family $\pi$. In fact Theorem \ref{th:intro Hk} is deduced from its Iwahori analog.

\subsection{Acknowledgments} The authors would like to thank Joakim F\ae rgeman for inspiring discussions that led to the uniqueness of the lifting of relative Lagrangians, and Marius Kj\ae rsgaard for pointing out an error in an earlier draft. We also thank the referee for helpful suggestions.

DN was partially supported by NSF DMS grant  \#2401178. ZY was partially supported by the Packard Fellowship and the Simons Investigator grant.

\section{Generalities on conic Lagrangians}

Let $k$ be an algebraically closed field of characteristic zero. 

\subsection{Conic Lagrangians for schemes}\label{ss:Lag}

 Let $M$ be a scheme that is equidimensional and smooth over $k$ (in particular $M$ is locally of finite type over $k$). Let $\pr_M: T^*M\to M$ be the projection.  

We call a subset $\L\subset T^{*}M$ constructible if for any open substack $U\subset M$ that is of finite type over $k$, the intersection $\L\cap T^{*}U$ is a constructible subset of $T^{*}U$. A constructible subset $\L\subset T^{*}M$ is {\em isotropic} if the canonical symplectic form $\om$ on $T^*M$ restricts to zero on the smooth locus of $\ov \L$ with the reduced structure. A reduced locally closed subscheme $\L\subset T^{*}M$ is a {\em Lagrangian} if it is isotropic and equidimensional with the same dimension as $M$. Equivalently, $\L$ is a Lagrangian if at every smooth point $\xi\in \L$, $T_\xi\L$ is a Lagrangian subspace of $T_\xi(T^*M)$.


A subscheme $\L\subset T^*M$ is conic if it is stable under the $\Gm$-action on $T^*M$ scaling cotangent fibers. We then have the notion of conic Lagrangians in $T^*M$.

For a smooth subscheme $Z\subset M$, let $T^*_ZM\subset T^*M$ denote the total space of the conormal bundle of $Z$ in $M$.

\begin{defn}\label{def:shadow}
    Let $\L\subset T^*M$ be a Lagrangian with the set of irreducible components denoted by $I$, and generic point $\y_\a$ for the irreducible component $\L_\a$ corresponding to $\a\in I$. The {\em shadow} of $\L$ is the set of (not necessarily closed) points $\{\z_\a=\pr_M(\y_\a)\}_{\a\in I}$ of $M$.
\end{defn}

\begin{lemma}\label{l:conic Lag}
    Let $\L\subset T^*M$ be a conic Lagrangian with shadow $\{\z_\a\}_{\a\in I}$. For each $\a\in I$, let $Z_\a\subset M$ be any smooth irreducible subscheme with generic point $\z_\a$. Then 
    \begin{equation}\label{isotropic conormal}
        \L\subset\ov{\bigcup_{\a\in I}T^*_{Z_\a}M}.
    \end{equation}
    If moreover $\L$ is closed, then we have an equality
    \begin{equation}\label{Lag conormal}
        \L=\ov{\bigcup_{\a\in I}T^*_{Z_\a}M}.
    \end{equation}
    In particular, a conic Lagrangian $\L\subset T^*M$ is determined by its shadow.
\end{lemma}
\begin{proof} For both statements, it suffices to treat the case where $\L$ is irreducible, in which case the shadow of $\L$ consists of a single point $\z\in M$. Let $Z$ be a smooth irreducible subscheme of $M$ with generic point $\z$. Let $\L'\subset \pr_M^{-1}(Z)\cap \L$ be the open dense locus where the projection map to $Z$ is smooth. In particular, $\L'$ is smooth. For \eqref{isotropic conormal}, it suffices to show that $\L'\subset T^{*}_{Z}M$.

Let $\xi\in \L'$ be a closed point and $x=\pr_M(\xi)\in Z$. The cotangent fiber $T^*_xM$ is identified with the kernel of $d_{\xi}\pr_{M}: T_\xi(T^*M)\to T^*_xM$. Restricting $d_{\xi}\pr_{M}$ to $T_{\xi}\L'$ we get a surjection $\pi: T_{\xi}\L'\surj T_{x}Z$. Since $\L'$ is isotropic, we must have $\ker(\pi)\subset (T_{x}Z)^{\bot}=(T^{*}_{Z}M)|_{x}$. On the other hand, since $\L'$ is conic, the punctured line through $\xi$ in $T^*_xM$ is also contained in $\L'$, hence $\ker(\pi)$ contains the vector $\xi$. These facts imply that $\xi\in (T^{*}_{Z}M)|_{x}$. This being true for all closed points $\xi$ of $\L'$, we conclude that $\L'\subset T^{*}_{Z}M$.


When $\L$ is a closed conic Lagrangian, the containment $\L'\subset T^{*}_{Z}M$ implies $\L=\ov{\L'}\subset \ov{T^*_ZM}$, which has to be an equality for dimension reasons.  This proves \eqref{Lag conormal}.
\end{proof}


\subsection{Lifting relative Lagrangians}
Let $\pi:M\to S$ be a smooth morphism between schemes that are smooth equidimensional over $k$. Let $T^*(\pi)$ be the total space of the relative cotangent bundle of $\pi$. We have a surjective bundle map $p: T^*M\to T^*(\pi)$. 

A locally closed reduced subscheme $\L^{\rel}\subset T^*(\pi)$ is called a {\em relative Lagrangian} if the following two conditions are satisfied:
\begin{itemize}
        \item For each geometric point $s\in S$, the reduced structure of $\L^{\rel}_{s}:=\L^{\rel}\cap T^*(M_s)$ is a Lagrangian in $T^*(M_s)$.
        
        \item $\L^{\rel}$ is equidimensional with the same dimension as $M$.
\end{itemize}

We may then talk about closed conic relative Lagrangians in $T^*(\pi)$. For a relative Lagrangian $\L^{\rel}\subset T^*(\pi)$, there is an obvious notion of shadow: it is the collection of points in $M$ that are images of generic points of $\L^{\rel}$. We have a relative version of Lemma \ref{l:conic Lag}.

\begin{lemma}\label{l:rel conic Lag}
    Let $\L^{\rel}\subset T^*(\pi)$ be a closed conic relative Lagrangian with shadow $\{\z_\a\}_{\a\in I}$. The following holds.
    \begin{enumerate}
        \item For each $\a\in I$, the image of $\z_\a$ in $S$ is a generic point of $S$. 
        \item For each $\a\in I$, let $Z_\a\subset M$ be any smooth irreducible subscheme with generic point $\z_\a$ such that $\pi_\a: Z_\a\to S$ is also smooth (such $Z_\a$ exists by part (1)). Then \begin{equation}\label{rel Lag conormal}
        \L^{\rel}=\ov{\bigcup_{\a\in I}T^*_{\pi_\a}(\pi)}.
        \end{equation}
        Here $T^*_{\pi_\a}(\pi)$ is the relative conormal bundle of $Z_\a$ in $M$ with respect to the maps to $S$: it is the fiberwise kernel of $T^*(\pi)|_{Z_\a}\to T^*(\pi_\a)$. In particular, a closed conic relative Lagrangian is determined by its shadow.
    \end{enumerate} 
\end{lemma}
\begin{proof}
    We claim that each irreducible component $L\subset \L^{\rel}$ is also a relative Lagrangian (certainly closed and conic). Indeed, the second condition on $\L^{\rel}$ implies $\dim L=\dim M$. The first condition for $\L^{\rel}$ implies that $L_s=L\cap T^*(M_s)$ is isotropic in $T^*(M_s)$ for all geometric points $s\in S$. Now Now $L_s$  is the fiber of $L \to S$ over $s$, hence its dimension is $\ge \dim L-\dim S=\dim M-\dim S=\dim M_s$. This forces the reduced structure of $L_s$ to be a Lagrangian in $T^*(M_s)$. Therefore $L$ is also a relative Lagrangian. Therefore we may replace $\L^{\rel}$ by $L$ and assume that $\L^{\rel}$ is an irreducible closed conic relative Lagrangian.


Replacing $S$ with the  connected component containing the image of $\L^{\rel}$, we may assume $S$ is connected hence irreducible. 

To prove (1), we claim that $\L^{\rel}\to S$ is dominant to a component of $S$. Suppose not, the fiber of $\L^{\rel}\to S$ over $s$, which is $\L^{\rel}_{s}$, would have dimension larger than $\dim M-\dim S=\dim M_s$, contradicting the assumption that $\L^{\rel}_{s}$ is isotropic. 

(2) Let $\z\in M$ be the shadow of $\L^{\rel}$, which maps to the generic point $\s$ of $S$. Therefore there exists a smooth subscheme $Z\subset M$ with generic point $\z$ such that $\pi_Z: Z\to S$ is also smooth. 

Let $\ov\s$ be a geometric generic point of $S$, and $\z_{\ov\s}$ be the generic point of $Z_{\ov\s}=\pi_Z^{-1}(\ov\s)
$. By assumption, $\L^{\rel}_{\ov\s}$ is a closed conic Lagrangian of $T^*(M_{\ov\s})$, therefore $\L^{\rel}_{\ov\s}=\ov{T^*_{Z_{\ov\s}}M_{\ov\s}}$ by Lemma \ref{l:conic Lag}. This implies $\L^{\rel}_{\s}=\ov{T^*_{Z_\s}M_{\s}}$. We have
\begin{equation*}
    \L^{\rel}=\ov{\L^{\rel}_{ \s}}=\ov{T^*_{Z_\s}M_{\s}}=\ov{T^*_{\pi_Z}(\pi)},
\end{equation*}
with closures all taken in $T^*(\pi)$. 
\end{proof}

The following statement concerns the existence and uniqueness of liftings of a relative Lagrangian to the cotangent bundle of $M$.

\begin{prop}\label{p:lifting Lag}
    Let $\L^{\rel}\subset T^*(\pi)$ be a closed conic relative Lagrangian. Then
    \begin{enumerate}
        \item There exists a unique closed conic Lagrangian $\L\subset T^*M$ with the following properties: $p(\L)\subset \L^\rel$ and is a dense there, and every irreducible component of $\L$ dominates an irreducible component of $\L^{\rel}$. We call $\L$ {\em the lifting} of $\L^{\rel}$.
        \item For the lifting $\L$ of $\L^{\rel}$, the map $p|_{\L}: \L\to \L^{\rel}$ is an isomorphism over an open dense subset of $\L^{\rel}$.
        \item The shadow of $\L^{\rel}$ and its lifting $\L$ are the same, and each point of the shadow maps to a generic point of $S$. 
    \end{enumerate}  
\end{prop}
\begin{proof}
(1) Uniqueness: if $\L$ and $\L'$ are closed conic Lagrangians in $T^*M$ both satisfying the said properties, then they have the same shadow, which is the shadow of $\L^{\rel}$. Therefore $\L=\L'$ by Lemma \ref{l:conic Lag}.

Existence: let $\{\z_\a\}_{\a\in I}$ be the shadow of $\L^{\rel}$, and let $Z_\a$ be as in Lemma \ref{l:rel conic Lag}(2). Take $\L=\ov{\cup_{\a\in I}T^*_{Z_\a}M}$, a closed conic Lagrangian in $T^*M$. We claim that $p(\L)\subset \L^{\rel}$ and is dense therein. This is clear since $p(T^*_{Z_\a}M)= T^*_{\pi_\a}(\pi)$, and union of the latter is dense in $\L^{\rel}$ by \eqref{rel Lag conormal}.

(2) We use notation from Lemma \ref{l:rel conic Lag}(2). By shrinking $Z_\a$, we may further assume $Z_\a$ are disjoint from each other. From the construction of the lifting $\L$, we have a Cartesian diagram
\begin{equation*}
    \xymatrix{\coprod_{\a\in I}T^*_{Z_\a}M \ar@{^{(}->}[r]\ar[d]_{\coprod p_\a} & \L\ar[d]^{p|_\L}\\
    \coprod_{\a\in I}T^*_{\pi_\a}(\pi) \ar@{^{(}->}[r] & \L^{\rel}}
\end{equation*}
It is clear that $p_\a$ is an isomorphism for each $\a$. This implies $p|_{\L}:\L\to \L^{\rel}$ is an isomorphism over the open dense subset $\coprod_{\a\in I}T^*_{\pi_\a}(\pi)$.

(3) follows from the uniqueness part of the proof of (1) and Lemma \ref{l:rel conic Lag}(1).
\end{proof}

\begin{exam}
   We illustrate that the map $p|_{\L}: \L\to \L^{\rel}$ may not be injective nor surjective in general. Let $Z\subset M$ be a closed smooth subscheme such that the map $\pi_Z=\pi|_Z: Z\to S$ is quasi-finite but ramified. Let $\L^{\rel}=T^*(\pi)|_{Z}\subset T^*(\pi)$. Fiberwise $\L^{\rel}_{ s}$ is the union of cotangent fibers of $M_s$ at $Z_s$ (which is a finite number of points). Therefore $\L^{\rel}$ is a closed conic relative Lagrangian.   It is clear that $\L=T^*_ZM$ is the lifting of $\L^{\rel}$. The map $\L\to \L^{\rel}$ over $z\in Z$ (with image $s\in S$) is the composition of linear maps $(T^*_ZM)_z\subset T^*_zM\surj T^*_z(M_s)$, which when $z$ is a ramified point for $\pi_Z$ is neither injective nor surjective. 
\end{exam}

Recall that a subset $\L\subset T^*M$ is {\em non-characteristic} with respect to $\pi$ if $\L\cap (T^*S\times_SM)$ is contained in the zero section $0_M$. We say $\L$ is {\em generically non-characteristic} with respect to $\pi$ if it contains an open dense subset that is non-characteristic with respect to $\pi$.

The following proposition gives an upper bound for certain Lagrangians in $T^*M$ by in terms of its image in the relative cotangent bundle. 
 
\begin{prop}\label{p:contained in lifting} Let $\L^{\rel}\subset T^*(\pi)$ be a closed conic relative Lagrangian. Let $\L'\subset T^*M$ be a conic Lagrangian that is generically non-characteristic with respect to $\pi$. Suppose $p(\L')\subset \L^\rel$, then $\L'$ is contained in the lifting of $\L^\rel$. 
\end{prop}
\begin{proof}
It suffices to treat the case where $\L'$ is irreducible. Replacing $\L'$ by an open dense subset, we may assume that $\L'$ is non-characteristic with respect to $\pi$. 

Let $\z\in M$ be the shadow of $\L'$, and $Z$ be a smooth subscheme of $M$ with generic point $\z$. By shrinking $Z$ we may assume $T^*_ZM\subset \L'$. Since $\L'$ is non-characteristic, $\pi_Z: Z\to S$ is dominant to a component of $S$, for otherwise, by further shrinking $Z$, we may assume $\pi(Z)\subset S'$ for a smooth subscheme $S'\subset S$ of lower dimension, which then implies that $\L'\supset (T^*_{S'}S)\times_{S'}Z$, contradicting the non-characteristic property of $\L'$. Shrink $Z$ further so that $\pi_Z$ is smooth. Then $p(\L')\supset T^*_{\pi_Z}(\pi)$. Since $p(\L')\subset \L^\rel$ by assumption, we have $T^*_{\pi_Z}(\pi)\subset \L^\rel$, and therefore $\z$ is also a shadow of $\L^\rel$. By the construction of the lifting $\L$ of $\L^\rel$, we have $\ov{T^*_ZM}\subset \L$. Since $\L'$ is irreducible by assumption, $T^*_ZM$ is dense in $\L'$, taking closure we obtain that $\L'\subset \ov{T^*_ZM}\subset \L$.
\end{proof}

\subsection{Conic Lagrangians for stacks}

Let $\frM$ be an algebraic stack that is smooth equidimensional over $k$. Let $T^*\frM$ denote the {\em classical} stack underlying the cotangent stack of $\frM$, the latter being a derived stack in general. A constructible subset $\L\subset T^*\frM$ is {\em isotropic} if for some smooth surjective map $f: M\to \frM$ from a scheme over $k$, the image of $\L\times_{\frM}M$ under the embedding
\begin{equation*}
    df:(T^*\frM)\times_{\frM}M\incl T^*M
\end{equation*}
is isotropic in $T^*M$ in the sense recalled in \S\ref{ss:Lag}. Similarly, a reduced locally closed substack $\L\subset T^*\frM$ is a {\em Lagrangian} if, under the same notation as above, the image of $\L\times_{\frM}M$ under $df$ 
is a Lagrangian in $T^*M$ in the sense recalled in \S\ref{ss:Lag}. It is easily checked that these properties are independent of the choice of smooth covering $f$ by schemes. The following criterion for Lagrangians is easily checked by going to a smooth cover by schemes:

\begin{lemma}\label{l:Lag crit}
    Let $\L\subset T^*\frM$ be a reduced locally closed substack that is isotropic. Then the following are equivalent:
    \begin{enumerate}
        \item $\L$ is a Lagrangian.
        \item $\L$ is equidimensional with the same dimension as $\frM$.
        \item $\L$ has local dimension $\ge \dim \frM$ everywhere.
    \end{enumerate}
\end{lemma}

Note however that the dimension of the classical stack $T^*\frM$ may be larger than $2\dim \frM$ in general.

The shadow of a Lagrangian $\L\subset T^*\frM$ is defined verbatim as in Definition \ref{def:shadow}, which is a subset of the underlying topological space of $\frM$.

For a point $\z$ in $\frM$, its closure $\ov{\{\z\}}$ is a reduced closed substack of $\frM$. For an open substack $\frZ\subset \ov{\{\z\}}$ that is smooth over $k$, the conormal bundle $T^*_\frZ\frM$ is defined, and is a conic Lagrangian in $T^*\frM$. Indeed, for a smooth covering $f:M\to \frM$ by a scheme, we have $(T^*_\frZ\frM)\times_{\frM}M=T^*_{Z}M$ (where $Z=f^{-1}(\frZ)$) as subschemes of $T^*M$. With this preparation, Lemma \ref{l:conic Lag} readily generalizes to stacks $\frM$.

If now $\frS$ is another smooth stack over $k$ and $\pi:\frM\to \frS$ is a smooth map of stacks. The (classical) relative cotangent stack $T^*(\pi)$ is defined. The notion of relative Lagrangian $\L^{\rel}\subset T^*(\pi)$ has a clear extension to this situation, as well as Lemma \ref{l:rel conic Lag}. Finally, Proposition \ref{p:lifting Lag} also extends to the situation of smooth maps $\pi:\frM\to \frS$ between stacks smooth over $k$. The proofs of all these statements go by sending Lagrangians $\L\subset T^*\frM$ to $\L\times_{\frM}M\subset T^*M$ for a smooth covering $M\to \frM$ by a scheme $M$. We omit details.

\subsection{Transport of Lagrangians}

Let $\frM_1$ and $\frM_2$ be smooth  equidimensional stacks over $k$. Let $f:\frM_{1}\to \frM_{2}$ be a morphism of finite type. Consider the correspondence
\begin{equation*}
\xymatrix{ T^{*}\frM_{1} & \frM_{1}\times_{\frM_{2}}T^{*}\frM_{2}\ar[r]^-{f^{\na}}\ar[l]_-{df} & T^{*}\frM_{2}
}
\end{equation*}
For a constructible subset $\L_{1}\subset T^{*}\frM_{1}$, define a constructible subset of $T^*\frM_2$
\begin{equation*}
\orr{f}(\L_{1})=f^{\na}( df^{-1}(\L_{1}))\subset T^{*}\frM_{2}.
\end{equation*}
Similarly, for a constructible subset $\L_{2}\subset T^{*}\frM_{2}$, define a constructible subset of $T^*\frM_1$
\begin{equation*}
\oll{f}(\L_{2})=df( (f^{\na})^{-1}(\L_{2}))\subset T^{*}\frM_{1}.
\end{equation*}

\begin{lemma}\label{l:Cart Lag} Consider a Cartesian diagram where all stacks involved are smooth equidimensional over~$k$
\begin{equation*}
\xymatrix{ \frM' \ar[d]^{f'}\ar[r]^{a} & \frM\ar[d]^{f}\\
\frN'\ar[r]^{b} & \frN}
\end{equation*}
Then for any subset $\L\subset T^{*}\frM$ we have an equality of constructible subsets of $T^{*}\frN'$
\begin{equation}\label{Cart Lag eq}
\oll{b}\orr{f}(\L)=\orr{f'}\oll{a}(\L)\subset T^{*}\frN'.
\end{equation}
\end{lemma}
\begin{proof} Let $p=f\c a=b\c f': \frM'\to \frN$. Let $Z=\frM'\times_{\frN}T^{*}\frN$. We have a natural map $\pi_{\frM}: Z\xr{a\times\id}\frM\times_{\frN}T^{*}\frN\xr{df}T^{*}\frM$. Similarly, we have a map $\pi_{\frN'}: Z\to T^{*}\frN'$. It is easy to see that
\begin{equation}\label{fa}
\orr{f'}\oll{a}(\L)=\pi_{\frN'}\pi_{\frM}^{-1}(\L).
\end{equation}
On the other hand, let $W=(\frM'\times_{\frN'}T^{*}\frN')\times_{T^{*}\frM'}(\frM'\times_{\frM}T^{*}\frM)$ (the maps to $T^{*}\frM'$ are $df'$ and $da$). Let $\nu_{\frM}: W\to T^{*}\frM$ and $\nu_{\frN'}: W\to T^{*}\frN'$ be the projections. Then one checks that
\begin{equation}\label{bf}
\oll{b}\orr{f}(\L)=\nu_{\frN'}\nu_{\frM}^{-1}(\L).
\end{equation}
Finally, we observe that the natural map $Z\to W$ (induced by $db$ and $df$) is an isomorphism, under which $\pi_{\frM}$ corresponds to $\nu_{\frM}$ and $\pi_{\frN'}$ corresponds to $\nu_{\frN'}$. Indeed this boils down to the fact that for any $x'\in \frM'$ with image $y',x$ and $y$ in $\frN', \frM$ and $\frN$ respectively, the map $(db,df): T^{*}_{y}\frN\to T^{*}_{y'}\frN'\times_{T^{*}_{x'}\frM'}T^{*}_{x}\frM$ is an isomorphism. The latter follows from the distinguished triangle of cotangent complexes $p^*\LL_{\frN}\to a^*\LL_{\frM}\op f'^*\LL_{\frN'}\to \LL_{\frM'}\to $ by taking the $0$th cohomology (all cotangent complexes are in degrees $\ge0$ by the smoothness assumption). The desired equality then follows by combining \eqref{fa} and \eqref{bf}.
\end{proof}

\begin{lemma}\label{l:pres iso} 
Let $\frM_1$ and $\frM_2$ be stacks that are smooth over $k$, and $f:\frM_1\to \frM_2$ be a morphism. Let $\L_{1}\subset T^{*}\frM_1$, $\L_{2}\subset T^{*}\frM_2$ be isotropic subsets. Then $\oll{f}(\L_{2})$ is an isotropic subset in $T^{*}\frM_1$, and $\orr{f}(\L_{1})$ is an isotropic subset in $T^{*}\frM_2$.
\end{lemma}
\begin{proof} We first consider the case where $\frM_1=X_1$ and $\frM_2=X_2$ are smooth schemes over $k$. In this case, the map $\L_f:=X_1\times_{X_2}T^*X_2\subset T^*X_1\times T^*X_2$ sending $(x_1, x_2, \xi_2)$ to  $(x_1, f^*(\xi_2), x_2, \xi_2)$ is a Lagrangian for the symplectic form $-\om_1+\om_2$ on $T^*X_1\times T^*X_2$ (where $\om_i$ is the canonical symplectic form on $X_i$). The statement for $\orr{f}(\L_{1})$ follows from \cite[Lemma 1]{Gin} applied to $\L_1\subset T^*X_1$ and $\L=\L_f\subset T^*X_1\times T^*X_2$. The statement for $\oll{f}(\L_2)$ follows from \cite[Lemma 1]{Gin} by switching the roles of $T^*X_1$ and $T^*X_2$.

Now we reduce the general case to the scheme case proved above. Let $\a_2: X_2\to \frM_2$ be a smooth surjection with $X_2$ a smooth scheme over $k$. Let $X_1:=\frM\times_{\frM_2}X_2$. Let $h: U\to X_1$ be a smooth surjection with $U$ a smooth scheme over $k$. We have a diagram where the square is Cartesian
\begin{equation*}
\xymatrix{U\ar[dr]^h\ar@/_1pc/[ddr]_{\a_1'}\ar@/^{1pc}/[drr]^{g'} \\
& X_1\ar[d]^{\a_1}\ar[r]^{g} & X_2\ar[d]^{\a_2}\\
& \frM_1\ar[r]^f & \frM_2}
\end{equation*}
To show $\oll{f}(\L_{2})$ is isotropic, it suffices to show $\oll{\a_1'}\oll{f}(\L_{2})\subset T^{*}U$ is isotropic. We have
\begin{equation*}
\oll{\a_{1}'}\oll{f}(\L_{2})=\oll{g'}\oll{\a_2}(\L_{2}).
\end{equation*}
Since $\L_{2}$ is isotropic and $\a_2$ is smooth, $\oll{\a_2}(\L_{2})$ is isotropic. By the scheme case applied to $g': U\to X_{2}$, we see that $\oll{g'}\oll{\a_2}(\L_{2})$ is isotropic, hence the same is true for $\oll{\a_{1}'}\oll{f}(\L_{2})$.

It remains to show that $\orr{f}(\L_{1})$ is isotropic. It suffices to show that $\oll{\a_2}\orr{f}(\L_{1})$ is isotropic. By Lemma \ref{l:Cart Lag}, we have
\begin{equation*}
\oll{\a_{2}}\orr{f}(\L_{1})=\orr{g}\oll{\a_{1}}(\L_{1}),
\end{equation*}
hence it suffices to show that $\orr{g}\oll{\a_{1}}(\L_{1})$ is isotropic. Let $L_{1}=\oll{\a_{1}}(\L_{1})$, which is isotropic in $T^{*}X_{1}$. Since $h$ is a smooth surjection, it is easy to see that $\orr{h}\oll{h}(L)=L$ for any subset $L\subset T^{*}X_{1}$. In particular, $\orr{h}\oll{h}(L_{1})=L_{1}$, hence
\begin{equation*}
\orr{g}(L_{1})=\orr{g}\orr{h}\oll{h}(L_{1})=\orr{g'}\oll{h}(L_{1}).
\end{equation*}
Now $\oll{h}(L_{1})$ is isotropic, and by the scheme case applied to $g'$, $\orr{g'}\oll{h}(L_{1})$ is isotropic, hence $\orr{g}(L_{1})=\orr{g}\oll{\a_{1}}(\L_{1})=\oll{\a_{2}}\orr{f}(\L_{1})$ is isotropic.
\end{proof}

\sss{Relative transport}\label{sss:rel transport}
Suppose $\frM_1$ and $\frM_2$ both admit smooth morphisms $\pi_i: \frM_i\to S$ to some base stack $S$, and $f:\frM_1\to \frM_2$ is a morphism over $S$. We may define the correspondence between relative cotangent stacks
\begin{equation*}
\xymatrix{ T^{*}(\pi_{1}) & \frM_{1}\times_{\frM_{2}}T^{*}(\pi_{2})\ar[r]^-{f^{\na,\rel}}\ar[l]_-{df^\rel} & T^{*}(\pi_{2})
}
\end{equation*}
and define the relative transport of constructible subsets in the relative cotangent stacks for $\L_i\subset T^*(\pi_i)$
\begin{eqnarray*}
\orr{f}^\rel(\L_{1})=f^{\na,\rel}( df^{\rel,-1}(\L_{1}))\subset T^{*}\frM_{2}, \\
\oll{f}^\rel(\L_{2})=df^\rel( (f^{\na,\rel})^{-1}(\L_{2}))\subset T^{*}\frM_{1}.
\end{eqnarray*}

\section{Global nilpotent cone: recollections}

We recall here two constructions of the global nilpotent cone for the moduli stack of $G$-bundles of a single curve. We also extend these constructions to the situation of $G$-bundles with Iwahori level structures.

\subsection{Groups} Let $G$ be a semisimple group over $k$, $B\subset G$ a Borel subgroup, with unipotent radical $N = [B, B]$, and Cartan quotient $H = B/N$. Let $\frg$, $\frb$, $\frn$, and $\frh$ denote the respective Lie algebras. Let $W$ denote the Weyl group, and $\frc = \frh\quot W =  \Spec k[\frh]^W\cong \Spec k[\frg]^G=\frg\quot G$.

We have  the characteristic polynomial map
$$
\xymatrix{
\chi:\frg/(G\times \Gm)  \ar[r] & (\frg\quot G)/\Gm \simeq (\frh\quot W)/\Gm=\frc/\Gm
}
$$ 
where the $G$-action on $\frg$ is the adjoint action, and the $\Gm$-action  is by  scaling. 
Note the $\Gm$-action on $\frh$ is also by  scaling, but 
 the weights of the induced $\Gm$-action on the affine space $\frc$ are the degrees of $W$.
 
Let $\wt \frg \to\frg$ be the Grothendieck-Springer alteration, and recall the canonical isomorphism
$\wt \frg/G \simeq \frb/B$ of adjoint quotients. We have the ordered eigenvalue map 
$$
\xymatrix{
\wt \chi:\wt \frg/(G \times \Gm) \simeq \frb/(B\times \Gm)  \ar[r] & \frh/\Gm 
}
$$ 
induced by $\frb\to \frh = \frb/\frn$.

\subsection{Marked smooth DM curves}\label{sss:glob nilp cone}
For applications to automorphic gluing, we will consider a slightly larger class of curves. Namely, let $X$ be a connected Deligne-Mumford stack that is proper and smooth of dimension $1$ over $k$, which is generically a scheme. Let $\omega_X$ be the canonical bundle of $X$.


Let $\Sigma\subset X$ be a finite set of closed points with trivial automorphism groups, and $\omega_X(\Sigma)$ the sheaf of $1$-forms with possibly simple poles at $\Sigma$. We will use the canonical isomorphism $\Res_{\Sigma}: \omega_X(\Sigma)|_\Sigma \simeq \cO_\Sigma$ given by taking the residues.

For a $G$-bundle $\cE$ over $X$ and a representation $V$ of $G$,  let $\cE(V)$ be the associated vector bundle on $X$.

For a stack $Y$ with $\Gm$-action, let $Y_{\om_{X}(\Sigma)}\to X$ denote the twist of $Y$ using the $\Gm$-torsor associated to $\om_{X}(\Sigma)$. Then the restriction of $Y_{\om_{X}(\Sigma)}$ over $\Sigma$ has a canonical trivialization $Y\times \Sigma$ using the residue map.

Consider the moduli stack
$$
\xymatrix{ 
\Bun_{G, N}(X, \Sigma) := \Maps((X, \Sigma), (\BB G, \BB N))
}
$$
classifying pairs $(\cE, \cE_{\Sigma,N})$ of a $G$-bundle $\cE$ on $X$ with an $N$-reduction $\cE_{\Sigma,N}$ along $\Sigma$.

Consider the moduli stack of Higgs bundles
$$
\xymatrix{
 \operatorname{Higgs}_{G, N}(X, \Sigma) = \Maps ((X, \Sigma), (\frg_{\om_{X}(\Sigma)}/(G\times \Gm),
  \frb/N)) 
}$$
classifying data $(\cE, \cE_{\Sigma,N}, \ph)$ consisting of a pair $(\cE, \cE_{\Sigma,N}) \in \Bun_{G, N}(X, \Sigma)$ along with a Higgs field 
$$
\xymatrix{
\ph\in \upH^0(X, \cE(\frg) \otimes \omega_X(\Sigma))
}
$$
such that
$$
\xymatrix{
\Res_{\Sigma}(\ph)  \in \upH^0(\Sigma, \cE_{\Sigma,N}(\frb)).
}
$$

The Killing form on $\frg$ gives a $G$-invariant identification $\frg^*\simeq \frg$, with induced identification $(\frg/\frn)^* \simeq \frb$. Serre duality provides  
 an isomorphism \footnote{We only consider classical stacks in this paper; the isomorphism continues to hold if both sides are given the canonical derived structures.}
$$
\xymatrix{
 T^*\Bun_{G, N}(X, \Sigma)  \simeq \operatorname{Higgs}_{G, N}(X, \Sigma) 
}$$
 
Consider the Hitchin base
$$
\xymatrix{
A_{G, N}(X, \Sigma)= 
\Maps((X, \Sigma), (\frc_{\omega_X(\Sigma)}, \frh))
}
$$ 
classifying pairs $(a, a_\Sigma)$ of a section
$a\in \upH^0(X,  \frc_{\omega_X(\Sigma)})$ with a lift of the residue $\Res_\Sigma (a)\in \upH^0(\Sigma, \frc)$ to $a_\Sigma\in \upH^0(\Sigma, \frh)$. If we choose homogeneous free generators $f_{1},\cdots, f_{r}$ of $\Sym(\frg^{*})^{G}$ of degrees $d_{1},\cdots, d_{r}$, then 
\begin{equation*}
A_{G, N}(X, \Sigma)\cong (\prod_{i=1}^{r}\cohog{0}{X,\om_{X}(\Sig)^{\ot d_{i}}})\times_{\frc}\frh.
\end{equation*}

Applying the characteristic polynomial map $\chi$ and ordered eigenvalue map $\wt \chi$ to Higgs fields provides the Hitchin system 
$$
\xymatrix{
H_{X,\Sig}^N:T^*\Bun_{G, N}(X, \Sigma)  \simeq  \operatorname{Higgs}_{G, N}(X, \Sigma)  \ar[r] &  A_{G, N}(X, \Sigma)
}
$$ 

The  {\em global nilpotent cone} is defined to be the fiber  of $H_{X,\Sig}^N$ over the zero point $0\in A_{G, N}(X, \Sigma)$
$$
\cN_{X,\Sigma}^{N}= (H^N_{X,\Sig})^{-1}(0) \subset T^*\Bun_{G, N}(X, \Sigma)
$$

When $\Sigma=\vn$, we denote the global nilpotent cone by $\cN_X$. 
It is proved by Laumon \cite{Lau} (in type $A$), Faltings \cite{Fa} and Ginzburg \cite{Gin} that (the reduced structure of) $\cN_X$ is a closed conic Lagrangian. The same is true when $\Sigma\ne\vn$, as we will show in Proposition \ref{p:two cones B} and Corollary \ref{c:NN}.

We may replace the $N$-reduction by $B$-reduction along $\Sigma$, and get a variant $\cN^B_{X,\Sigma}$ as the zero fiber of the Hitchin map 
\begin{equation*}
H^B_{X,\Sig}: T^*\Bun_{G,B}(X,\Sigma)\to A_{G,B}(X,\Sig)=
\Maps((X, \Sigma), (\frc_{\omega_X(\Sigma)}, \{0\})).
\end{equation*}
Using homogeneous free generators $f_{1},\cdots, f_{r}$ of $\Sym(\frg^{*})^{G}$, we have
\begin{equation*}
A_{G,B}(X,\Sig)\cong \prod_{i=1}^{r}\cohog{0}{X,\om^{\ot d_{i}}((d_{i}-1)\Sig)}.
\end{equation*}



We record here the standard dimension calculation.
\begin{lemma}\label{l:base dim}
Assume
\begin{equation}\label{hyp}
\deg \om_{X}+|\Sigma|>0.
\end{equation}
Then we have
\begin{equation*}
\dim\Bun_{G,N}(X,\Sigma)=\dim A_{G,N}(X,\Sig).
\end{equation*}
If moreover $G$ is semisimple, then 
\begin{equation*}
    \dim\Bun_{G,B}(X,\Sigma)=\dim A_{G,B}(X,\Sig).
\end{equation*}
\end{lemma}
The assumption \eqref{hyp} is used to ensure $\cohog{1}{X, \om_{X}(\Sig)^{\ot d_{i}}}=0$ in the calculation of $\dim A_{G,N}(X,\Sig)$.

\begin{remark}
    The dimension equalities in Lemma \ref{l:base dim} are necessary conditions for the global nilpotent cones $\cN^N_{X,\Sigma}$ and $\cN^B_{X,\Sigma}$ to be Lagrangians. For deeper level structures such as ones defined by principal congruence subgroups, one can still define the global nilpotent cone as the zero fiber of the Hitchin map, but typically the dimension equality analogous to those in Lemma \ref{l:base dim} will not hold. For example, for the first congruence level structure along $\Sigma$ (i.e., adding a trivialization along $\Sigma$), the dimension of the resulting moduli stack $\Bun_{G,1}(X,\Sig)$ is larger than the dimension of the corresponding Hitchin base. In particular, the global nilpotent cones for general level structures are not Lagrangian . 
\end{remark}



\subsection{Alternative construction: Eisenstein cone} 
Ginzburg \cite{Gin} gives an alternative construction of the global nilpotent cone as the microlocal shadow of the Eisenstein series construction for automorphic sheaves. It turns out that this alternative construction generalizes more easily to the family situation. We recall Ginzburg's construction.

Consider first a smooth connected projective curve $X$ over $k$. Let $\Bun_{B}(X)$ be the moduli stack of $B$-bundles on $X$ and $ p: \Bun_{B}(X)\to \Bun_{G}(X)$ be the  induction map. Consider the associated Lagrangian correspondence
\begin{equation*}
\xymatrix{T^{*}\Bun_{B}(X) &  \Bun_{B}(X)\times_{ \Bun_{G}(X)}T^{*} \Bun_{G}(X)\ar[l]_-{dp}\ar[r]^-{p^{\na}} & T^{*}\Bun_{G}(X) }
\end{equation*}

Let $0_{\Bun_{B}(X)}$ denote the zero section of $T^*\Bun_{B}(X)$.

\begin{defn} The {\em Eisenstein cone} $\cN^{\Eis}_{X}$ of $X$ is the constructible subset $T^{*}\Bun_{G}(X)$
\begin{equation*}
\cN^{\Eis}_{X}=\orr{p}(0_{\Bun_{B}(X)}) = p^{\na}((dp)^{-1}(0_{\Bun_{B}(X)})).
\end{equation*}
\end{defn}

By definition, a point $(\cE,\ph)\in T^{*}\Bun_{G}(X)$ lies in $\cN^{\Eis}_{X}$ if and only if there exists a $B$-reduction $\cE_B$ of $\cE$ such that $\ph\in \cohog{0}{X, \cE_{B}(\frn)\ot\om_{X}}$.


It is not a priori clear that $\cN^{\Eis}_{X}$ is a constructible subset of $T^{*}\Bun_{G}(X)$, because the map $p: \Bun_{B}(X)\to \Bun_{G}(X)$ is not of finite type. The constructibility of $\cN^{\Eis}_{X}$ will be proved in the next lemma by identifying it with the global nilpotent cone $\cN_{X}$.



\begin{lemma}\label{l:two cones} Assume $\deg\om_X>0$ and $G$ is semisimple \footnote{These assumptions will be removed in the more general Proposition  \ref{p:two cones B}.}. Then the subset $\cN^{\Eis}_{X}$ of $T^{*}\Bun_{G}(X)$ coincides with the underlying set of the global nilpotent cone $\cN_{X}$,  i.e.~the zero-fiber of the Hitchin map. In particular, $\cN^{\Eis}_{X}$ is a closed conic Lagrangian of $T^{*}\Bun_{G}(X)$.
\end{lemma}
\begin{proof}
This is essentially the argument of Ginzburg \cite[Lemma 5]{Gin}. By the constructibility of $\cN^{\Eis}_{X}$, it suffices to show that $\cN^{\Eis}_{X}(K)=\cN_{X}(K)$ for any algebraically closed field $K$ containing $k$. By making a base change from $k$ to $K$, we may assume $K=k$. 

First we show that $\cN^{\Eis}_{X}(k)\subset\cN_{X}(k)$. Let $\cE_{B}$ be a $B$-torsor over $X$ and let  $\cE$ be the induced $G$-torsor. The cotangent space of $\Bun_{B}(X)$ at $\cE_{B}$ is $\cohog{0}{X, \cE_{B}(\frb^{*})\ot\om_{X}}$. The differential of the induction map $p: \Bun_{B}(X)\to \Bun_{G}(X)$ at $\cE$ is induced by the restriction map $\cE(\frg^{*})=\cE_{B}(\frg^{*})\to \cE_{B}(\frb^{*})$ 
\begin{equation*}
\xymatrix{  dp:    \cohog{0}{X, \cE(\frg^{*})\ot\om_{X}} \ar[r] &     \cohog{0}{X, \cE_{B}(\frb^{*})\ot\om_{X}}
}
\end{equation*}
Under the Killing form on $\frg$, $\ker(\frg^{*}\to \frb^{*})$ is identified with 
$\frn$, Therefore $\ker(dp)_{\cE} = \cohog{0}{X, \cE_{B}(\frn)\ot\om_{X}}$, which visibly consists of nilpotent Higgs fields.

Now we prove the other inclusion $\cN_{X}(k)\subset\cN^{\Eis}_{X}(k)$. Let $(\cE,\ph)\in \cN_{X}\subset T^{*}\Bun_{G}$, where $\ph\in\cohog{0}{X,\cE(\frg)\ot\om_{X}}$ is a nilpotent Higgs field. Restricting to the generic point $\eta$ of $X$, and choosing trivializations of $\cE_{\eta}$ and $\om_{X,\eta}$, $\ph_{\eta}$ gives a nilpotent element in the Lie algebra $\frg(F)$, where $F=k(X)$ is the function field of $X$. By the Jacobson-Morozov theorem, $e$ extends to an $\sl_{2}$-triple $(e,h,f)$ in $\frg(F)$, and gives a canonical parabolic subgroup $P$ (i.e., $\Lie P$ is the sum of $\ge0$ weight spaces of $\ad(h)$) of $G$ defined over $F$, such that $\ph_{\eta}\in \frn_{P}$  (the nilpotent radical of $\Lie(P)$). 
The parabolic subgroup $P$ can be viewed as a $P$-reduction of the trivialized $\cE_{\eta}$, which then gives a $P$-reduction $\cE_{P}$ of $\cE$ by saturation. By construction, 
\begin{equation*}
\ph\in\cohog{0}{X,\cE_{P}(\frn_{P})\ot\om_{X}}.
\end{equation*}
Now let $\cE_{B}$ be an arbitrary reduction of $\cE_{P}$ to $B$, then $\cE_B(\frn)\supset \cE_{P}(\frn_{P})$, hence $\ph$ lies in $\cohog{0}{X,\cE_{B}(\frn)\ot\om_{X}}$ as well, and therefore $(\cE,\ph)\in \cN^{\Eis}_{X}(k)$.

Finally we show that $\cN^{\Eis}_{X}$ is a Lagrangian. By Lemma \ref{l:pres iso}, $\cN^{\Eis}_{X}$ is isotropic. It remains to show that $\cN^{\Eis}_{X}$ has dimension $\ge\dim\Bun_{G}(X)$ everywhere, which follows from Lemma \ref{l:base dim}.
\end{proof}

\subsection{Eisenstein cone with Iwahori level structure}\label{sss:univ cone B}
Let $X$ be a smooth projective curve over $k$, and $\Sigma\subset X$ be finitely many closed points where we impose $B$ or $N$ reductions on $G$-bundles.
 
Let $\Bun_{G,1}(X,\Sigma)$ (resp. $\Bun_{B,1}(X,\Sigma)$) be the moduli of $G$-bundles (resp. $B$-bundles) on $X$ with a trivialization at $\Sigma$. Since $\Bun_{G,1}(X,\Sigma)\to \Bun_{G}(X)$ (resp. $\Bun_{B,1}(X,\Sigma)\to \Bun_{B}(X)$) is a $G^\Sigma$-torsor (resp. $B^{\Sigma}$-torsor), the cotangent bundle $T^{*}\Bun_{G,1}(X,\Sigma)$ (resp. $T^{*}\Bun_{B,1}(X,\Sigma)$) is $G^\Sigma$-equivariant (resp. $B^{\Sigma}$-equivariant) hence descends to a vector bundle $\Om_{G, X,\Sigma}$ (resp. $\Om_{B, X,\Sigma}$) over $\Bun_{G}(X)$ (resp. $\Bun_{B}(X)$), whose fiber at $\cE_G$  (resp. $\cE_{B}$) is 
 $\cohog{0}{X, \cE_{G}(\frg^{*})\ot\om_{X}(\Sigma)}$  (resp. $\cohog{0}{X, \cE_{B}(\frb^{*})\ot\om_{X}(\Sigma)}$).

Let $p_{\Sigma}: \Bun_{B,1}(X,\Sigma)\to \Bun_{G,1}(X,\Sigma)$ be the induction map, giving rise to  the correspondence
\begin{equation}\label{Om}
\xymatrix{\Om_{B,X,\Sigma} & \Bun_{B}(X)\times_{\Bun_{G}(X)}\Om_{G,X,\Sigma}\ar[l]_-{dp_{\Sigma}}\ar[r]^-{p^{\na}_{\Sig}} & \Om_{G,X,\Sigma}}
\end{equation}
Define 
\begin{equation*}
\un\cN^{\Eis}_{X,\Sigma}=p^{\na}_{\Sig}((dp_{\Sigma}^{-1}(0_{\Bun_{B}(X)})))\subset \Om_{G,X,\Sigma}
\end{equation*}
We have natural maps
\begin{equation*}
q: T^{*}\Bun_{G,B}(X,\Sigma)\to \Om_{G, X,\Sigma}, \quad q': T^{*}\Bun_{G,N}(X,\Sigma)\to \Om_{G, X,\Sigma}
\end{equation*}
given by the differential of the projections
$\Bun_{G,1}(X,\Sigma)\to \Bun_{G,B}(X,\Sigma)$ and $\Bun_{G,1}(X,\Sigma)\to \Bun_{G,N}(X,\Sigma)$.
Set
\begin{eqnarray*}
\cN^{\Eis, B}_{X,\Sigma}=q^{-1}(\un\cN^{\Eis}_{X,\Sigma}),\quad \cN^{\Eis, N}_{X,\Sigma}=(q')^{-1}(\un\cN^{\Eis}_{X,\Sigma}).
\end{eqnarray*}

By definition, a geometric point $(\cE,\cE_{\Sig, B}, \ph)\in T^{*}\Bun_{G,B}(X,\Sig)(K)$ lies in $ \cN^{\Eis,B}_{X,\Sig}(K)$ if and only if there exists a $B$-reduction $\cE_B$ of $\cE$ such that $\ph\in \cohog{0}{X_{K}, \cE_{B}(\frn)\ot\om_{X}(\Sig)}$. Note that the $B$-reduction $\cE_B$ need not be the same as the $B$-reduction $\cE_{\Sig,B}$ when restricted to $\Sig$. Similarly, $(\cE,\cE_{\Sig, N}, \ph)\in \cN^{\Eis,N}_{X,\Sig}(K)$ if and only if there exists a $B$-reduction $\cE_B$ of $\cE$ such that $\ph\in \cohog{0}{X_{K}, \cE_{B}(\frn)\ot\om_{X}(\Sig)}$.

\sss{Another description of $\cN^{\Eis,B}_{X, \Sigma}$} 
We abbreviate $\Bun_{G,B}(X,\Sigma)$ by $\Bun_{G,B}$, and $\Bun_{B}(X)$ by $\Bun_{B}$, etc. We have an isomorphism
\begin{equation*}
\Bun_{B}\times_{\Bun_{G}}\Bun_{G,B}\simeq \Bun_{B,1}\twtimes{B^{\Sigma}}(G/B)^{\Sigma},
\end{equation*}
from which we get two projections
\begin{equation}\label{BunB1 BunGB}
\xymatrix{(B\bs G/B)^{\Sigma} & \ar[l]_-{\b}\Bun_{B,1}\twtimes{B^{\Sigma}}(G/B)^{\Sigma}\simeq\Bun_{B}\times_{\Bun_{G}}\Bun_{G,B}\ar[r]^-{\a} & \Bun_{G,B}}
\end{equation}

\begin{lemma}\label{l:nilcone B alt}
We have
\begin{equation*}
\cN^{\Eis,B}_{X,\Sigma}=\orr{\a}\oll{\b}(T^{*}(B\bs G/B)^{\Sigma}).
\end{equation*}
\end{lemma}
\begin{proof}
Consider the commutative diagram
\begin{equation}\label{sq Om}
\xymatrix{T^{*}(\Bun_{B,1}\twtimes{B^{\Sigma}}(G/B)^{\Sigma})\ar[d]^{q_{B}} & \Bun_{B}\times_{\Bun_{G}}T^{*}\Bun_{G,B}\ar[d]^{\wt q}\ar[l]_-{d\a}\ar[r]^-{\a^{\na}} & T^{*}\Bun_{G,B}\ar[d]^{q}\\
\Om_{B,X,\Sigma} & \Bun_{B}\times_{\Bun_{G}}\Om_{G,X,\Sigma}\ar[l]_-{dp_{\Sigma}}\ar[r]^-{p^{\na}_{\Sig}} & \Om_{G,X,\Sigma}}
\end{equation}
Here the top row is the Lagrangian correspondence for $\a$ and the bottom row is \eqref{Om}. The map $q_{B}$ is the composition
\begin{equation*}
T^{*}(\Bun_{B,1}\twtimes{B^{\Sigma}}(G/B)^{\Sigma})\incl (T^{*}\Bun_{B,1})\times^{B^{\Sig}} T^{*}(G/B)^{\Sigma}\to (T^{*}\Bun_{B,1})/B^{\Sig}=\Om_{B,X,\Sigma}.
\end{equation*}
Note the right square of \eqref{sq Om} is Cartesian, which implies $q^{-1}p^{\na}_{\Sig}(Z)=\a^{\na}\wt q^{-1} (Z)$ for any subset $Z\subset \Bun_{B}(X)\times_{\Bun_{G}(X)}\Om_{G,X,\Sigma}$. 
We have
\begin{equation*}
\cN^{\Eis,B}_{X,\Sig}=q^{-1}p^{\na}_{\Sig}dp_{\Sig}^{-1}(0_{\Bun_{B}(X)})=\a^{\na}\wt q^{-1} dp_{\Sig}^{-1}(0_{\Bun_{B}(X)})=\a^{\na}d\a^{-1}q_{B}^{-1}(0_{\Bun_{B}(X)})=\orr{\a}q_{B}^{-1}(0_{\Bun_{B}(X)}).
\end{equation*}
Since
\begin{equation*}
q_{B}^{-1}(0_{\Bun_{B}(X)})=\oll{\b}(T^{*}(B\bs G/B)^{\Sigma}),
\end{equation*}
the lemma follows.
\end{proof}

Now we have the following generalization of Lemma \ref{l:two cones}.

\begin{prop}\label{p:two cones B} 
The subset $\cN^{\Eis,B}_{X, \Sigma}$ of $T^{*}\Bun_{G, B}(X, \Sigma)$ coincides with the underlying set of the global nilpotent cone $\cN^B_{X, \Sigma}$. Moreover, $\cN^{\Eis,B}_{X,\Sigma}$ is a closed conic Lagrangian in $T^{*}\Bun_{G, B}(X, \Sigma)$.
\end{prop}
\begin{proof}
The same argument of Lemma \ref{l:two cones} shows that $\un\cN^{\Eis}_{X,\Sigma}$ is the same as the underlying set of the zero fiber of the Hitchin map for $\Om_{G,X,\Sigma}$, i.e., $\un\cN^{\Eis}_{X,\Sigma}(k)$ consists of $(\cE,\ph)$ where $\cE$ is a $G$-torsor over $X$ and $\ph\in \cohog{0}{\cE(\frg^{*})\ot \om_{X}(\Sig)}$ is nilpotent (at the generic point).  

We show $\cN^{\Eis,B}_{X,\Sigma}(K)=\cN^{B}_{X,\Sigma}(K)$ for any algebraically closed field $K\supset k$. By base change it suffices to consider the case $K=k$. By definition, $\cN^{\Eis,B}_{X,\Sigma}(k)$ consists of $(\cE, \cE_{\Sig, B}, \ph)$ such that $(\cE, \ph)\in \Om_{G,X,\Sig}(k)$. But this is the same as requiring $\ph$ to be nilpotent, which shows that $\cN^{\Eis,B}_{X,\Sigma}(k)=\cN^{B}_{X,\Sig}(k)$. 
This shows that $\cN^{\Eis,B}_{X,\Sigma}=\cN^{B}_{X,\Sigma}$ as subsets of $T^{*}\Bun_{G, B}(X, \Sigma)$. In particular, $\cN^{\Eis,B}_{X,\Sigma}$ is closed.

Next we show that $\cN^{\Eis,B}_{X,\Sigma}$ is isotropic. For $\l\in \xcoch(T)$ let $\Bun_{B}^{\l}$ be the open closed substack of $\Bun_{B}$ where the induced $T$-bundle has degree $\l$. Let $\Bun^{\l}_{B,1}$ be the preimage of $\Bun^{\l}_{B}$ in $\Bun_{B,1}$. Restricting \eqref{BunB1 BunGB} we get a similar diagram
\begin{equation*}
\xymatrix{(B\bs G/B)^{\Sigma} & \ar[l]_-{\b^{\l}}\Bun^{\l}_{B,1}\twtimes{B^{\Sigma}}(G/B)^{\Sigma}\simeq\Bun^{\l}_{B}\times_{\Bun_{G}}\Bun_{G,B}\ar[r]^-{\a^{\l}} & \Bun_{G,B}}
\end{equation*}
Let $W^{\l}_{\Sig}=\orr{\a^{\l}}\oll{\b^{\l}}(T^{*}(B\bs G/B)^{\Sigma})$. The entire classical cotangent bundle $T^{*}(B\bs G/B)^{\Sigma}$ is isotropic. Indeed, it suffices to consider the case $|\Sigma|=1$, in which case $T^{*}(B\bs G/B)\cong B\bs \L$, where $\L\subset T^*(G/B)$ is the union of conormals to $B$-orbits, hence a Lagrangian. Therefore $W^{\l}_{\Sig}$ is also isotropic by Lemma \ref{l:pres iso}. Note that, even with Lemma \ref{l:nilcone B alt}, we cannot apply Lemma \ref{l:pres iso} directly to conclude that $\cN^{\Eis,B}_{X,\Sigma}$ is isotropic because $\a$ is not of finite type. By Lemma \ref{l:nilcone B alt}, $\cN^{\Eis,B}_{X,\Sigma}=\cup_{\l\in\xcoch(T)}W^{\l}_{\Sig}$. Let $U\subset \Bun_{G,B}$ be any open substack of finite type over $k$. Then $\cN^{\Eis,B}_{X,\Sigma}\cap T^{*}U=\cN^{B}_{X,\Sig}\cap T^{*}U$ is also of finite type over $k$, hence covered by finitely many $W^{\l}_{\Sig}$. Since each $W^{\l}_{\Sig}$ is isotropic, $\cN^{\Eis,B}_{X,\Sigma}\cap T^{*}U$ is isotropic; this being true for all $U\subset \Bun_{G,B}$, we conclude that $\cN^{\Eis,B}_{X,\Sigma}$ is isotropic.


%


It remains to show that $\cN^{\Eis,B}_{X,\Sigma}=\cN^{B}_{X,\Sigma}$ is a Lagrangian. Since $\cN^{B}_{X,\Sigma}$ is isotropic, by Lemma \ref{l:Lag crit}, to show it is a Lagrangian, it suffices to show that $\cN^{B}_{X,\Sigma}$ has dimension $\ge\dim\Bun_{G,B}(X,\Sigma)$ everywhere. When Assumption \eqref{hyp} holds and $G$ is semisimple, the dimension estimate follows from Lemma \ref{l:base dim}. Still assuming \eqref{hyp} but without assuming $G$ semisimple, we consider the smooth surjective map $p: \Bun_{G,N}(X,\Sigma)\to \Bun_{G,B}(X,\Sigma)$. We have $\cN^{N}_{X,\Sigma}=\oll{p}\cN^{B}_{X,\Sigma}$. Since $p$ is smooth, it suffices to show that $\cN^{N}_{X,\Sigma}$ has dimension $\ge\dim\Bun_{G,N}(X,\Sigma)$ everywhere, which follows from Lemma \ref{l:base dim}. In general, without any assumptions, we may enlarge $\Sigma$ to $\Sigma'\supset \Sigma$ so that \eqref{hyp} holds and consider the projection $\nu: \Bun_{G,B}(X,\Sig')\to \Bun_{G,B}(X,\Sig)$. It is easy to check that $\oll{\nu}\cN^{B}_{X,\Sigma}=\cN^{B}_{X,\Sigma'}$. Since $\nu$ is smooth, and we already proved that $\cN^{B}_{X,\Sigma'}$ is a Lagrangian in $T^*\Bun_{G,B}(X,\Sig')$, $\cN^B_{X,\Sig}$ is also a Lagrangian in $T^*\Bun_{G,B}(X,\Sig)$. This finishes the proof of the lemma. 
\end{proof}

\begin{cor}\label{c:NN} The subset $\cN^{\Eis,N}_{X, \Sigma}$ of $T^{*}\Bun_{G, N}(X, \Sigma)$ coincides with the underlying set of the global nilpotent cone $\cN^N_{X, \Sigma}$. Moreover, $\cN^{\Eis,N}_{X}$ is a closed conic Lagrangian in $T^{*}\Bun_{G, N}(X, \Sigma)$.
\end{cor}
\begin{proof}
Let $p: \Bun_{G,N}(X,\Sigma)\to \Bun_{G,B}(X,\Sigma)$ be the projection. We have $\cN^{N}_{X,\Sigma}=\oll{p}\cN^{B}_{X,\Sigma}$. Since $p$ is  smooth, $\oll{p}$ takes a closed conic Lagrangian to a closed conic Lagrangian. 
\end{proof}

\section{Universal nilpotent cone}

\subsection{Marked relative curves}\label{ss:rel curve}

Let $S$ be an algebraic stack locally of finite type over $k$, and $\pi:\frX\to S$ a proper smooth relative DM curve. Let $\frX^0 \subset \frX$ be the locus where $\pi$ is schematic and smooth. We assume $\frX^{0}$ is dense in each fiber of $\pi$. Let $\sigma\subset \frX^0$ be a closed substack that is finite \'etale over $S$. Let $\omega_\pi$ be the relative sheaf of $1$-forms along the fibers of $\pi$, which is a line bundle on $\frX$, and $\omega_\pi(\sigma)$ the twisting allowing simple poles along $\s$.

Given a map of stacks $s:S'\to  S$, with base change
$\pi_s:\frX_s = \frX \times_S S'\to S'$, $\sigma_s = \sigma \times_S S'$, we have  canonical isomorphisms $\omega_{\pi}|_{\frX_s} \simeq \omega_{\pi_s }$,
 $\omega_{\pi}(\sigma)|_{\frX_s} \simeq \omega_{\pi_s}(\sigma_s)$.
 Note also the canonical isomorphism $\omega_{\pi}(\sigma)|_\sigma \simeq \cO_\sigma$ given by the residue map.
 
At various places we will assume that the condition \eqref{hyp} is satisfied for each fiber of $(\frX, \s)\to S$.

\subsection{Relative moduli of bundles}\label{ss:rel Hitchin}
Consider the moduli stack $\Bun_{G, N}(\pi, \sigma)$ over $S$ whose $S'$-points 
classifies triples $(s, \cE, \cE_{\sigma, N})$ where $s:S'\to S$, $\cE$ is a $G$-bundle on $\frX_s = \frX \times_S S'$,
and $\cE_{\sigma, N}$ is an $N$-reduction of $\cE$ along $\sigma_s = \sigma \times_S S'$. Replacing $N$-reduction by $B$-reduction we get a similarly defined moduli stack $\Bun_{G, B}(\pi, \sigma)$ over $S$.

We have forgetful maps
\begin{equation}\label{forg level}
\xymatrix{\Bun_{G, N}(\pi, \sigma)\ar[r]^-{q} & \Bun_{G, B}(\pi, \sigma)\ar[r]^-{r} & \Bun_{G}(\pi)}
\end{equation}
where $q$ is a $R_{\s/S}(H)$-torsor (where $R_{\s/S}(-)$ stands for Weil restriction along the finite \'etale map $\s\to S$) and $r$ is a smooth proper \'etale locally trivial fibration whose fibers are products of flag varieties of $G$. 

Consider the Higgs  moduli stack
$
 \operatorname{Higgs}_{G, N}(\pi, \sigma)
 $
classifying data $(s, \cE, \cE_{\sigma, N}, \ph)$ of a triple $(s, \cE, \cE_{\sigma, N}) \in \Bun_{G, N}(\pi, \sigma)$ along with a Higgs field 
$$
\xymatrix{
\ph\in \upH^0(\frX_s, \cE(\frg)\otimes \omega_{\pi_s}(\sigma_s))
}
$$
such that its residue along $\s$, a priori a section of $\cE(\frg)|_{\s_{s}}$, satisfies
$$
\Res_{\s}(\ph) \in \upH^0(\sigma_s, \cE_{\sigma, N}(\frb)).
$$

Let $\Pi^{N}_{\s}: \Bun_{G, N}(\pi, \sigma)\to S$  denote the natural projection, which is smooth. Now assume $S$ is smooth over $k$, so that $\Bun_{G, N}(\pi, \sigma)$ is smooth over $k$. Let $T^{*}(\Pi^{N}_{\s})$ be the (classical) relative cotangent stack of $\Pi^{N}_{\s}$. Consider projection to the relative cotangent bundle
\begin{equation}\label{ppi}
P^{N}_{\pi,\s}:T^*\Bun_{G, N}(\pi, \sigma) \to  T^*(\Pi^{N}_{\s}). 
\end{equation}



Using the Killing form to identify $(\frg/\frn)^*$ with $\frb$,  Grothendieck-Serre  duality provides  
 an isomorphism
$$
\xymatrix{
 T^*(\Pi^{N}_{\s})   \simeq \operatorname{Higgs}_{G, N}(\pi, \sigma) 
}$$

Let $\frc_{\omega_{\pi}(\sigma)}$ denote the $\omega_\pi(\sigma)$-twist of $\frc$ (as a scheme over $\frX$ with $\Gm$-action), and  consider the Hitchin base $ A_{G, N}(\pi, \sigma) $ 
whose $S'$-points classify triples $(s, a, a_\sigma)$ where $s:S'\to S$, and
$a\in \upH^0(\frX_s,  \frc_{\omega_{\pi_s}(\sigma_s)})$ with a lift of its residue $a_\sigma\in \upH^0(\sigma_s, \frh)$.

Applying the characteristic polynomial map $\chi$ to the Higgs fields and projecting $\Res_{\s}(\ph)$ along $\frb\to \frh$, we obtain the Hitchin system
 $$
\xymatrix{
H^{N}_{\pi,\s}: T^*(\Pi^{N}_{\s})   \simeq  \operatorname{Higgs}_{G, N}(\pi, \sigma)  \ar[r] &  A_{G, N}(\pi, \sigma)
}
$$ 

Let $0_S\subset A_{G, N}(\pi, \sigma)$ be the fiberwise zero section over $S$, which is a copy of $S$.  The {\em relative universal nilpotent cone} is defined to be the preimage of $0_S$ under $H^{N}_{\pi,\s}$
$$
 \cN^{\rel,N}_{\pi,\s} = (H^{N}_{\pi,\s})^{-1}(0_S) \subset T^*(\Pi^{N}_{\s})
$$ 

Replacing $N$-reductions by $B$-reductions, we get the map $\Pi^{B}_{\s}: \Bun_{G,B}(\pi,\s)\to S$,  the Hitchin map $ T^{*}\Bun_{G,B}(\pi,\s)\to A_{G,B}(\pi,\s)$ and the relative global nilpotent cone
$$
 \cN^{\rel,B}_{\pi,\s} = (H^{B}_{\pi,\s})^{-1}(0_S) \subset T^*(\Pi^{B}_{\s}).
$$

\begin{lemma}\label{l:rel nilp cone} The relative global nilpotent cone $\cN^{\rel,N}_{\pi,\s}$ is a closed conic relative Lagrangian in $T^*(\Pi^{N}_{\s})$. The fiber of $\cN^{\rel,N}_{\pi,\s}$ over a geometric point $s\in S$ is the global nilpotent cone for $\Bun_{G,N}(X_{s},\s_{s})$ as defined in \S\ref{sss:glob nilp cone}. Similar statements hold for $\cN^{\rel,B}_{\pi,\s}$.
\end{lemma}
\begin{proof} We prove the statements for $\cN^{\rel,N}_{\pi,\s}$ and the statements for $\cN^{\rel,B}_{\pi,\s}$ are proved similarly.  

Abbreviate $\cN^{\rel,N}_{\pi,\s}$ by $\cN^{\rel}$. All statements are clear except possibly for the equidimensionality of $\cN^{\rel}$. By the fiberwise description of $\cN^{\rel}$ we see that $\dim \cN^{\rel}\le \dim \Pi+\dim S=\dim \Bun_{G,N}(\pi,\s)$. To show the inequality in the other direction, we first assume \eqref{hyp} holds fiberwise. Consider the dimensions of the source and target of the relative Hitchin map $H=H^{N}_{\pi,\s}$. The stack $T^*(\Pi)$, being the classical stack of the derived cotangent bundle of the smooth $\Bun_{G,N}(\pi,\s)$, has dimension $\ge 2\dim \Bun_{G,N}(\pi,\s)-\dim S=\dim \Bun_{G,N}(\pi,\s)+\dim \Pi$. Under the fiberwise assumption \eqref{hyp}, the Hitchin base $A_{G,N}(\pi,\s)$ is smooth over $S$ with relative dimension equal to the relative dimension of $\Pi$ by Lemma \ref{l:base dim}, hence $0_S$ is a local complete intersection in $A_{G,N}(\pi,\s)$ of codimension $\dim \Pi$. Thus we conclude that each irreducible component of $\cN^{\rel}=H^{-1}(0_S)$ has dimension $\ge \dim \Bun_{G,N}(\pi,\s)+\dim \Pi-\dim \Pi=\dim \Bun_{G,N}(\pi,\s)$. Combined with the upper bound we conclude that $\cN^{\rel}$ is equidimensional with dimension equal to $\dim \Bun_{G,N}(\pi,\s)$.

Finally we drop the fiberwise assumption \eqref{hyp}. For the purpose of proving dimension lower bound, we are free to localize $S$ in the \'etale topology. Therefore we may add finite \'etale multi-sections to the family $\pi$ away from $\s$, and enlarge $\s$ to $\s'$ so that \eqref{hyp} holds fiberwise. The smooth surjection $\nu: \Bun_{G,N}(\pi,\s')\to \Bun_{G,N}(\pi,\s)$ induces a relative transport operation $\oll{\nu}^{\rel}$ between constructible subsets of cotangent bundles (see \S\ref{sss:rel transport}), so that $\oll{\nu}^{\rel}\cN^{\rel,N}_{\pi,\s}=\cN^{\rel,N}_{\pi,\s'}$. Since $\nu$ is smooth, and $\cN^{\rel,N}_{\pi,\s'}$ is equidimensional with the same dimension as $\Bun_{G,N}(\pi,\s')$ by the prior paragraph, we conclude that $\cN^{\rel,N}_{\pi,\s}$ is also equidimensional with the same dimension as $\Bun_{G,N}(\pi,\s)$.
\end{proof}


By Proposition \ref{p:lifting Lag}, we have a unique lifting $\cN^{N}_{\pi,\s}$ of $\cN^{\rel,N}_{\pi,\s}$ to $T^{*}\Bun_{G,N}(\pi,\s)$ which is a closed conic Lagrangian. However, the construction of the lifting in {\em loc. cit.} is not explicit enough as it involves taking closure. Our goal is to give an explicit construction of $\cN^{N}_{\pi,\s}$ as the family version of the Eisenstein cone reviewed in \S\ref{ss:Eis}, which allows us to prove stronger properties for the lifting.

\subsection{Universal Eisenstein cone}\label{ss:Eis} 

We first consider the case where $\s=\vn$. By the fiberwise assumption \eqref{hyp}, $G$ is semisimple and $\frX$ is hyperbolic fiberwise. We denote $\Pi: \Bun_{G}(\pi)\to S$ and $P_{\pi}:  T^{*}\Bun_{G}(\pi)\to  T^*(\Pi)$ the natural projections. Denote the relative Hitchin map in this case by $H_{\pi}$, and the global relative nilpotent cone by $\cN^{\rel}_\pi = H^{-1}_{\pi}(0_{S}) \subset T^*(\Pi)$. Denote the unique closed conic Lagrangian lifting of $\cN^{\rel}_{\pi}$ by $\cN_{\pi}\subset T^{*}\Bun_{G}(\pi)$.

Let $\frp: \Bun_{B}(\pi)\to \Bun_{G}(\pi)$ be the induction map. Note that $\Bun_{B}(\pi)\to S$ is also smooth. We have the  correspondence
\begin{equation}\label{Lag frp}
\xymatrix{T^{*}\Bun_{B}(\pi) & \Bun_{B}(\pi)\times_{\Bun_{G}(\pi)}T^{*}\Bun_{G}(\pi)\ar[l]_-{d\frp}\ar[r]^-{\frp^{\na}} & T^{*}\Bun_{G}(\pi)
}
\end{equation}

\begin{defn}\label{d:univ nilp cone} Define the {\em universal Eisenstein cone} $\cN^{\Eis}_{\pi}\subset T^{*}\Bun_{G}(\pi)$ to be
\begin{equation*}
\cN^{\Eis}_{\pi}=\orr{\frp}(0_{\Bun_{B}(\pi)})=\frp^{\na}(d\frp)^{-1}(0_{\Bun_{B}(\pi)}).
\end{equation*}
\end{defn}

The main result of this subsection is:
\begin{theorem}\label{th:Eis cone closed} Assume $G$ is semisimple and \eqref{hyp} holds fiberwise in the situation $\s=\vn$ \footnote{These assumptions will be removed in the more general Theorem \ref{th:univ cone Iw}.}. Then
\begin{enumerate}
    \item The universal Eisenstein cone $\cN^{\Eis}_{\pi}$ is a closed conic Lagrangian inside $T^{*}\Bun_{G}(\pi)$. 
    \item The projection $P_{\pi}$ restricts to a universal homeomorphism $P^{\Eis}_\pi: \cN^{\Eis}_{\pi}\to \cN^{\rel}_{\pi}$.
    \item In particular, $\cN^{\Eis}_{\pi}$ is the unique lifting of the relative nilpotent cone $\cN^{\rel}_{\pi}$ as in Proposition \ref{p:lifting Lag}.
\end{enumerate}
\end{theorem}

We deduce some consequences before giving the proof.

\begin{cor}\label{c:univ cone non-char} The universal Eisenstein cone $\cN^{\Eis}_{\pi}$ is non-characteristic with respect to the projection $\Pi:\Bun_{G}(\pi)\to S$, i.e. its intersection with  $d\Pi(T^*S) \subset T^* \Bun_{G}(\pi)$ lies in the zero-section.
\end{cor}
\begin{proof} Suppose $\xi$ is a $k$-point of $\cN^{\Eis}_{\pi}\cap d\Pi(T^*S)$ over $(s,\cE)\in \Bun_{G}(\pi)(k)$, where $s\in S(k)$. Then $\xi$ maps to the zero covector $(s,\cE,0)$ in the relative cotangent $T^{*}(\Pi)$. On the other hand, $0_{\Bun_{G}(\pi)}\subset \cN^{\Eis}_{\pi}$ because $\frp$ is surjective. Therefore $\xi$ and $(s,\cE,0)$ have the same image in $T^{*}(\Pi)$.  By Theorem \ref{th:Eis cone closed}(2), we must have $\xi=(s,\cE,0)$, i.e., $\cN^{\Eis}_{\pi}\cap d\Pi(T^*S)$ lies in the zero-section.
\end{proof}

\sss{Compatibility with base change}\label{sss:bc}
Let $S'$ be a stack smooth over $k$, and $t: S'\to S$ be a map of stacks. Let $\pi':\frX'=S'\times_{S}\frX\to S'$ be the base change of $\pi$. Note the natural identification
$S'\times_{S}\Bun_{G}(\pi)\simeq\Bun_{G}(\pi')$. Let
\begin{equation*}
\xymatrix{
\th: \Bun_{G}(\pi')\ar[r] &  \Bun_{G}(\pi)
}
\end{equation*}
be the natural projection. Passing to cotangent bundles we have the map
\begin{equation*}
\xymatrix{
d\th: \Bun_{G}(\pi')\times_{\Bun_{G}(\pi)}T^{*}\Bun_{G}(\pi)\cong S'\times_{S} T^{*}\Bun_{G}(\pi) \ar[r] &  T^{*}\Bun_{G}(\pi').
}
\end{equation*}

\begin{lemma}\label{l:Eis cone bc} The map $d\th$ restricts to a universal homeomorphism $S'\times_{S}\cN^{\Eis}_{\pi}\isom \cN^{\Eis}_{\pi'}$. In particular, we have $\cN^{\Eis}_{\pi'}=\oll{\th}(\cN^{\Eis}_{\pi})$.
\end{lemma}
\begin{proof}
Let $\frp': \Bun_{B}(\pi')\to \Bun_{G}(\pi')$ be the induction map for the family $\pi'$. We have a commutative diagram where both squares are Cartesian
\begin{equation*}
\xymatrix{      S'\times_{S}T^{*}\Bun_{G}(\pi) \ar[r]^-{d\th} &  T^{*}\Bun_{G}(\pi')  \\
\Bun_{B}(\pi')\times_{\Bun_{G}(\pi)}T^{*}\Bun_{G}(\pi) \ar[d]^{\id_{S'}\times d\frp}\ar[u]_{\id_{S'}\times\frp^{\na}} \ar[r]^-{d\th|_{\Bun_{B}(\pi')}} &  \Bun_{B}(\pi')\times_{\Bun_{G}(\pi')}T^{*}\Bun_{G}(\pi')\ar[d]^{d\frp'}\ar[u]_{\frp'^{\na}} \\
 S'\times_{S}T^{*}\Bun_{B}(\pi) \ar[r]^-{d\th_{B}} &  T^{*}\Bun_{B}(\pi')  }
\end{equation*}
Since $S'\times_{S}0_{\Bun_{B}(\pi)}$ maps bijectively to $0_{\Bun_{B}(\pi')}$ as subsets of the bottom row, a diagram chasing shows that $S'\times_{S}\cN^{\Eis}_{\pi}$ maps into $\cN^{\Eis}_{\pi'}$. 

We then have a commutative diagram
\begin{equation*}
\xymatrix{S'\times_{S}\cN^{\Eis}_{\pi}\ar[d]_{\id_{S'}\times P_{\pi}} \ar[r]^-{d\th^{\Eis}}  & \cN^{\Eis}_{\pi'}\ar[d]^{P_{\pi'}}\\
S'\times_{S}\cN^{\rel}_{\pi} \ar[r]^-{d\th^{\rel}} & \cN^{\rel}_{\pi'}}
\end{equation*}
Equip all stacks above with the reduced structure. The map $d\th^{\rel}$ is clearly an isomorphism. By Theorem \ref{th:Eis cone closed}, both vertical maps are universal homeomorphisms, hence the same is true for $d\th^{\Eis}: S'\times_{S}\cN^{\Eis}_{\pi}\to \cN^{\Eis}_{\pi'}$.
\end{proof}

The rest of the subsection is devoted to the proof of Theorem \ref{th:Eis cone closed}.

\begin{lemma}\label{l:NEis bij Nrel}
The projection $P_{\pi}$ maps $\cN^{\Eis}_{\pi}$ to $\cN^{\rel}_{\pi}$, inducing a bijection on geometric points. 
\end{lemma}
\begin{proof} We need to check that for any algebraically closed field $K\supset k$, $P_\pi$ induces a bijection on $K$-points $\cN^{\Eis}_{\pi}(K)\isom \cN^{\rel}_{\pi}(K)$. Base changing the situation from $k$ to $K$, we only need to consider the case $K=k$. Let $s\in S(k)$. Applying Lemma \ref{l:Cart Lag} to the Cartesian diagram
\begin{equation*}
\xymatrix{\Bun_{B}(\frX_{s}) \ar[r]\ar[d] & \Bun_{B}(\pi)\ar[d]\\
\Bun_{G}(\frX_{s}) \ar[r] & \Bun_{G}(\pi)
}
\end{equation*}
and the zero section of $T^{*}\Bun_{B}(\pi)$, we conclude that the image of $\cN^{\Eis}_{\pi}$ in $T^{*}\Bun_{G}(\frX_{s})$ is $\cN^{\Eis}_{\frX_{s}}=\cN_{\frX_{s}}$. Therefore the image of $\cN^{\Eis}_{\pi}$ under $P_{\pi}$ is precisely $\cN^{\rel}_{\pi}$.

It remains to show that $P_{\pi}: \cN^{\Eis}_{\pi}\to \cN^{\rel}_{\pi}$ is injective on $k$-points. Let $(s,\cE,\ph)\in \cN^{\rel}_{\pi}(k)$ (which means $s\in S(k)$ and $(\cE,\ph)$ is in the global nilpotent cone of $T^{*}\Bun_{G}(X_{s})$). Let $(s,\cE,\wt\ph)$ and $(s,\cE,\wt\ph')$ be two $k$-points of $\cN^{\Eis}_{\pi}$ lifting $(s,\cE,\ph)$, so that $\wt\ph,\wt\ph'\in T^{*}_{(s,\cE)}\Bun_{G}(\pi)$. Note that $\th:=\wt\ph-\wt\ph'\in T^{*}_{s}S$, and our goal is to show that $\th=0$. 

By definition, $(s,\cE,\wt\ph)\in \cN^{\Eis}_{\pi}(k)$ means that there exists a $B$-reduction $\cE_{B}$ of $\cE$ such that $d\frp_{(s,\cE_{B})}(\wt\ph)=0$, where $d\frp_{(s,\cE_{B})}: T^{*}_{(s,\cE)}\Bun_{G}(\pi)\to T^{*}_{(s,\cE_{B})}\Bun_{B}(\pi)$ is the cotangent map of $\frp$ at $(s,\cE_{B})$. Similarly, there exists another $B$-reduction $\cE'_{B}$ of $\cE$ such that $d\frp_{(s,\cE'_{B})}(\wt\ph')=0$.

We claim that $d\frp_{(s,\cE'_{B})}(\wt\ph)=0$. The idea is based on the proof of the connectedness of Springer fibers by Spaltenstein. Let $w$ be the relative position between the two $B$-reductions $(\cE_{B}, \cE'_{B})$ at the generic point of $\frX_{s}$. Upon choosing a generic trivialization of $\cE$, $\cE_{B}$ and $\cE_{B'}$ become two Borel subgroups $(B,B')$ of $G$ in relative position $w$ defined over $F=k(\frX_{s})$.  Fixing a reduced word $w=s_{i_{1}}s_{i_{2}}\cdots s_{i_{\ell}}$, and there is a unique chain of Borel subgroups $B=B_{0},B_{1},\cdots, B_{\ell}=B'$ of $G$ (a priori over $\ov F$) such that the relative position of $(B_{j-1}, B_{j})$ is $s_{i_{j}}$ for $1\le j\le \ell$. The uniqueness of such a chain implies that all $B_{j}$ are defined over $F$, hence by saturation they determine a chain of $B$-reductions of $\cE$
\begin{equation*}
\cE_{B}=\cE^{0}_{B}, \cE^{1}_{B},\cdots, \cE^{\ell-1}_{B}, \cE^{\ell}_{B}=\cE'_{B}
\end{equation*}
such that the relative position of $(\cE^{j-1}_{B},\cE^{j}_{B})$ at the generic point of $\frX_{s}$ is $s_{i_{j}}$ for $1\le j\le \ell$. We claim that the Higgs field $\ph\in \cohog{0}{\cE^{j}_{B}(\frn)\ot \om}$ for all $j$. Indeed, using the chosen generic trivialization of $\cE$, $\ph$ becomes a nilpotent element $\g\in \frg(F)$, and by assumption $\g$ lies in the nilradicals of both $B$ and $B'$. In other words, $\exp(\g)$ fixes $B$ and $B'$. By the uniqueness of the chain $(B_{j})$ connecting $B$ and $B'$, each $B_{j}$ is also fixed by $\exp(\g)$, hence $\g$ lies in the nilradical of $B_{j}$ for each $j$, i.e.,  $\ph\in \cohog{0}{\cE^{j}_{B}(\frn)\ot \om}$ for all $j$. Now we show that $d\frp_{(s,\cE^{j}_{B})}(\wt\ph)=0$ by induction on $j$. For $j=0$ this is true by assumption. Suppose $d\frp_{(s,\cE^{j-1}_{B})}(\wt\ph)=0$. Since $\cE^{j-1}_{B}$ and $\cE^{j}_{B}$ are in relative position $s_{i_{j}}$, $\cE^{j-1}_{B}(\frn)\cap \cE^{j}_{B}(\frn)=\cE_{P_{j}}(\frn_{P_{j}})$, where $P_{j}$ is the standard parabolic subgroup whose Levi has roots $\pm\a_{i_{j}}$, and $\cE_{P_{j}}$ is the $P_{j}$-torsor induced from $\cE^{j-1}_{B}$, which is also the $P_{j}$-torsor induced from $\cE^{j}_{B}$. Since $\ph\in \cohog{0}{\cE^{j-1}_{B}(\frn)\ot \om}$ and $\ph\in \cohog{0}{\cE^{j}_{B}(\frn)\ot \om}$, we see that
\begin{equation}\label{ph in nPj}
\ph\in \cohog{0}{\cE_{P_{j}}(\frn_{P_{j}})\ot \om}.
\end{equation}
Let $\frp_{P_{j}}: \Bun_{P_{j}}(\pi)\to \Bun_{G}(\pi)$ be the canonical map. Consider its cotangent map which fits into a commutative diagram
\begin{equation*}
\xymatrix{ & & T^{*}_{(s,\cE^{j-1}_{B})}\Bun_{B}(\pi) \\
T^{*}_{(s,\cE)}\Bun_{G}(\pi) \ar[r]^{d\frp_{P_{j}, (s,\cE_{P_{j}})}} \ar@/^1.5pc/[urr]^{d\frp_{(s,\cE^{j-1}_{B})}}\ar@/_1.5pc/[drr]_{d\frp_{(s,\cE^{j}_{B})}} &  T^{*}_{(s,\cE_{P_{j}})}\Bun_{P_{j}}(\pi)\ar[ur]\ar[dr] \\
 & & T^{*}_{(s,\cE^{j}_{B})}\Bun_{B}(\pi)
}
\end{equation*}
The fact that \eqref{ph in nPj} holds implies that $d\frp_{P_{j}, (s,\cE_{P_{j}})}(\wt\ph)$ lies in the image of $T^{*}_{s}S\incl T^{*}_{(s,\cE_{P_{j}})}\Bun_{P_{j}}(\pi)$, which must then be zero because its further image in $T^{*}_{(s,\cE^{j-1}_{B})}\Bun_{B}(\pi)$ is zero by the induction hypothesis that $d\frp_{(s,\cE^{j-1}_{B})}(\wt\ph)=0$. Therefore, $d\frp_{P_{j}, (s,\cE_{P_{j}})}(\wt\ph)=0$, hence $d\frp_{(s,\cE^{j}_{B})}(\wt\ph)=0$. This completes the inductive argument and shows in particular $d\frp_{(s,\cE'_{B})}(\wt\ph)=0$.

Finally, $\th=\wt\ph-\wt\ph'\in T^{*}_{s}S$ satisfies $d\frp_{(s,\cE'_{B})}(\th)=d\frp_{(s,\cE'_{B})}(\wt\ph)-d\frp_{(s,\cE'_{B})}(\wt\ph')=0$. Since $T_{s}^{*}S\to T^{*}_{(s,\cE'_{B})}\Bun_{B}(\pi)$ is injective, we must have $\th=0$. The proof of the lemma is complete.
\end{proof}

\sss{$\cN^{\Eis}_{\pi}$ is a Lagrangian}\label{sss:NEis const} 


Note that it is not a priori clear that $\cN^{\Eis}_{\pi}$ is constructible, because $\frp: \Bun_{B}(\pi)\to \Bun_{G}(\pi)$ is not of finite type.

For $\l\in \xcoch(T)$, let $\Bun_{B}^{\l}(\pi)\subset \Bun_{B}(\pi)$ be the open and closed substack corresponding to those $B$-bundles whose induced $T$-bundle have degree $\l$. Let $\frp^{\l}: \Bun^{\l}_{B}(\pi)\to \Bun_{G}(\pi)$ be the restriction of $\frp$, and let $\cW^{\l}=\orr{\frp^{\l}}(0_{\Bun^{\l}_{B}(\pi)})$, which is a constructible subset of $T^{*}\Bun_{G}(\pi)$ since $\frp^{\l}$ is of finite type. By Lemma \ref{l:pres iso}, each $\cW^{\l}$ is conic and isotropic.

We show that $\cN^{\Eis}_{\pi}$ is constructible, conic and isotropic. Let $\cU\subset \Bun_{G}(\pi)$ be an open substack of finite type over $k$ and let $\Pi_{\cU}: \cU\to S$ be the restriction of $\Pi$. The map $P_{\pi, \cU}: \cN^{\Eis}_{\pi}\cap T^{*}\cU\to \cN^{\rel}_{\pi}\cap T^{*}(\Pi_{\cU})$ is a bijection by Lemma \ref{l:NEis bij Nrel}. Therefore $\cup_{\l\in \xcoch(T)}P_{\pi}(\cW^{\l})$ covers $\cN^{\rel}_{\pi}\cap T^{*}(\Pi_{\cU})$. Since $\cN^{\rel}_{\pi}\to \Bun_{G}(\pi)$ is of finite type, $\cN^{\rel}_{\pi}\cap T^{*}(\Pi_{\cU})$ is of finite type over $k$. Therefore there is a finite subset $\L_{\cU}\subset \xcoch(T)$ such that $\cup_{\l\in \L_{\cU}}P_{\pi}(\cW^{\l})$ already covers $\cN^{\rel}_{\pi}\cap T^{*}(\Pi_{\cU})$. Since $P_{\pi}:\cN^{\Eis}_{\pi}\to \cN^{\rel}_{\pi}$ is a bijection on points, $\cup_{\l\in \L_{\cU}}\cW^{\l}$ covers $\cN^{\Eis}_{\pi}\cap T^{*}\cU$. This implies that $\cN^{\Eis}_{\pi}\cap T^{*}\cU$ is constructible, conic and isotropic. This being true for any finite type $\cU$, we conclude that $\cN^{\Eis}_{\pi}$ is isotropic.



Finally, by Lemma \ref{l:rel nilp cone}, $\cN^{\rel}_{\pi}$ is equidimensional with dimension equal to $\dim\Bun_{G}(\pi)$. Since $\cN^{\Eis}_{\pi}\to \cN^{\rel}_{\pi}$ is bijective, we conclude that $\dim\cN^{\Eis}_{\pi}=\dim\cN^{\rel}_{\pi}=\dim\Bun_{G}(\pi)$. Hence $\cN^{\Eis}_{\pi}$ is a Lagrangian.

\sss{Closedness of $\cN^{\Eis}_{\pi}$: the case of $\GL_{n}$}\label{sss:NEis closed GLn}
 
It remains to show that $\cN^{\Eis}_{\pi}$ is closed. This will be eventually proved in full generality without assuming \eqref{hyp} fiberwise or $G$ semisimple. Here we first consider the case $G=\GL_{n}$.


For $\l=(\l_{1},\cdots, \l_{n})\in \ZZ^{n}$, following Laumon \cite{LauEis}, the forgetful map $\frp^{\l}: \Bun^{\l}_{B}(\pi)\to \Bun_{G}(\pi)$ admits a relative compactification
\begin{equation*}
\ov\frp^{\l}: \ov\Bun^{\l}_{B}(\pi)\to \Bun_{G}(\pi).
\end{equation*}
Here, $\ov\Bun^{\l}_{B}(\pi)$ classifies $(s,\cE_{\bu}=\{\cE_{i}\}_{1\le i\le n})$, where $s\in S$, $\cE_{i}$ is a rank $i$ vector bundle over $\frX_{s}$ of degree $\l_{i}-\l_{i-1}$, together with maps $\cE_{1}\incl \cE_{2}\incl\cdots \cE_{n-1}\incl \cE_{n}$ that are injective as maps of coherent sheaves along each fiber of $\pi:\frX\to S$ (but not necessarily injective on stalks). The open locus $\Bun^{\l}_{B}(\pi)$ corresponds to where each $\cE_{i-1}\incl \cE_{i}$ is a subbundle. Note that 
$\ov\Bun^{\l}_{B}(\pi)$ is smooth over $S$ by a deformation-theoretic calculation.

Let $\ov\cW^{\l}=\orr{\ov\frp^{\l}}(0_{\ov\Bun^{\l}_{B}(\pi)})\subset T^{*}\Bun_{G}(\pi)$. Since $\ov\frp^{\l}$ is proper, $\ov\cW^{\l}$ is closed in $T^{*}\Bun_{G}(\pi)$. We also abbreviate $\cN^{\Eis,\l}_{\pi}=\orr{\frp^{\l}}(0_{\Bun^{\l}_{B}(\pi)})$ by $\cW^{\l}$.


\begin{lemma}\label{l:ov W}
We have
\begin{equation}\label{union W}
\cN^{\Eis}_{\pi}=\bigcup_{\l\in \ZZ^{n}}\ov\cW^{\l}.
\end{equation}
\end{lemma}
\begin{proof}

By construction, $\cN^{\Eis}_{\pi}=\cup_{\l}\cW^{\l}\subset \cup_{\l}\ov\cW^{\l}$. On the other hand, for any $\l\in \ZZ^{n}$, we show that
\begin{equation}\label{ov W estimate}
\ov\cW^{\l}\subset \bigcup_{\l'\ge \l}\cW^{\l'},
\end{equation}
where $\l'\ge\l$ means $\l'-\l$ is a sum of positive coroots of $\GL_{n}$, i.e., $\l'_{1}+\cdots+\l'_{i}\ge \l_{1}+\cdots+\l_{i}$ for all $i=1,\cdots, n$, and the equality holds for $i=n$.

To show the inclusion \eqref{ov W estimate} for $K$-valued points for any algebraically closed field $K\supset k$, we can make a base change from $k$ to $K$ and reduce to showing the inclusion for $k$-points. Let $(s,\cE,\xi)\in \ov\cW^{\l}(k)$, which means $s\in S(k)$, $\cE$ is a $G$-bundle over $\frX_{s}$, and $\xi\in T^{*}_{(s,\cE)}\Bun_{G}(\pi)$, and there exists a subsequence of subsheaves $\cE_{1}\incl \cE_{2}\incl \cdots\incl \cE_{n-1}$ of $\cE=\cE_{n}$ such that $\cE_{i}$ has rank $i$ and degree $\l_{1}+\cdots+\l_{i}$ for $1\le i\le n$, and $d\ov\frp^{\l}(\xi)=0\in T^{*}_{(s,\cE_{\bu})}\ov\Bun^{\l}_{B}(\pi)$.

Let $\cE'_{i}$ be the saturation of $\cE_{i}$ in $\cE$. Let $\l'\in\ZZ^{n}$ be defined such that $\deg(\cE'_{i})=\l'_{1}+\cdots+\l'_{i}$ for $1\le i\le n$. Then $\l'\ge \l$, and $(s,\cE'_{\bu})\in \Bun^{\l'}_{B}(\pi)(k)$ maps to $(s,\cE)\in\Bun_{G}(\pi)(k)$. We will construct an auxiliary stack $Z$ that fits into a commutative diagram
\begin{equation}\label{Z over BunB}
\xymatrix{Z \ar[r]^-{\a}\ar[d]_{\b} & \ov\Bun^{\l}_{B}(\pi)\ar[d]^{\ov\frp^{\l}}\\
\Bun^{\l'}_{B}(\pi)\ar[r]^-{\frp^{\l'}} & \Bun_{G}(\pi)}
\end{equation}
such that $\b$ is smooth. Moreover, we will construct a point $z\in Z(k)$ mapping to $(s,\cE'_{\bu})\in \Bun^{\l'}_{B}(\pi)(k)$ and to $(s,\cE_{\bu})\in \ov\Bun^{\l}_{B}(\pi)(k)$. Consider the cotangent maps
\begin{equation*}
\xymatrix{T^{*}_{z}Z & \ar[l]_-{d\a}  T^{*}_{(s,\cE_{\bu})}\ov\Bun^{\l}_{B}(\pi)\\
T^{*}_{(s,\cE'_{\bu})}\Bun^{\l'}_{B}(\pi)\ar[u]_{d\b} & \ar[l]_-{d\frp^{\l'}}\ar[u]^{d\ov\frp^{\l}} T^{*}_{(s,\cE)}\Bun_{G}(\pi)}
\end{equation*}
By assumption, we have
\begin{equation*}
0=d\a\c d\ov\frp^{\l}(\xi)= d\b\c d\frp^{\l'}(\xi)=0.
\end{equation*}
Since $\b$ is smooth, $d\b$ is injective, we must have $d\frp^{\l'}(\xi)=0$, i.e., $(s,\cE,\xi)\in \cW^{\l'}(k)$. This proves \eqref{ov W estimate}.

Now the construction of $Z$ and $z\in Z(k)$. For this purpose we can replace $S$ by a smooth scheme over $k$ that covers the point $s$. Let $D_{i}$ be the divisor associated to the torsion sheaf quotient $\cE'_{i}/\cE_{i}$ (support weighted with multiplicities). Let $\Sig_{s}\subset \frX_{s}$ be the support of $\sum_{i=1}^{n-1} D_{i}$. Upon replacing $S$ by an \'etale cover, we may assume that $\Sig_{s}$ extends to a collection of disjoint sections of $\pi: \frX\to S$ indexed by a finite set $\Sig$, so that $\Sig_{s}=\{\t(s)|\t\in \Sig\}$. Let $m$ be the maximum of the multiplicities that appear in the divisors $D_{1},\cdots, D_{n-1}$. Let $\t^{(m)}$ be the $m$th infinitesimal neighborhood of $\t(S)\subset \frX$. Upon making a further smooth base change for $S$, we may choose a trivialization $i_{\t}: \t^{(m)}\cong S\times_{k} \Spec R_{m}$ over $S$, where $R_{m}=k[t]/(t^{m})$.

Let $\Bun_{B,m}(\pi,\Sig)$ be the moduli stack of $B$-bundles over fibers of $\pi$ together with a trivialization along the $m$th infinitesimal neighborhood $\t^{(m)}$ of each section $\t\in \Sig$. 

For each $\t\in \Sig$, let $d_{i}(\t)$ be the multiplicity of $D_{i}$ at $\t(s)$. Let $Y_{m,\t}$ be the moduli space of $M_{\bu}=(M_{i})_{1\le i\le n}$ where each $M_{i}\subset R_{m}^{\op i}$ is an $R_{m}$-submodule such that $R_{m}^{\op i}/M_{i}$ has length $d_{i}(\t)$ (for $i=n$ we declare $M_{n}=R_{m}^{\op n}$), and that $M_{i-1}\subset M_{i}$ compatibly with the inclusion $R_{m}^{\op i-1}\incl R_{m}^{\op i}$ (inclusion of the first $(i-1)$-coordinates, $2\le i\le n$). It is clear that $Y_{m,\t}$ is a projective scheme over $k$. 
Let $Y_{m,\Sig}=\prod_{\t\in \Sig}Y_{m,\t}$ and let
\begin{equation*}
Z':=\Bun^{\l'}_{B,m}(\pi,\Sig)\times Y_{m,\Sig}.
\end{equation*}
We have a canonical projection $\b': Z'\to \Bun^{\l'}_{B,m}(\pi,\Sig)\to \Bun^{\l'}_{B}(\pi)$. We also have a canonical map $\a': Z'\to \ov\Bun_{B}^{\l}(\pi)$ defined as follows. A point of $Z'$ is a tuple 
$(s', \cF_{\bu}, \{\io(\t)\}_{\t\in \Sig}, \{M_{\bu}(\t)\}_{\t\in \Sig})$, where $(s',\cF_{\bu})\in \Bun^{\l'}_{B}(\pi)$, $\io(\t)$ is the system of trivializations $R^{\op i}_{m}\cong \cF_{i}|_{\t^{(m)}(s')}$ compatible with inclusions, and $M_{\bu}(\t)\in Y_{m,\t}$. The map $\a'$ sends such a tuple to $(s',\cG_{\bu})$ where $\cG_{i}\subset \cF_{i}$ is the preimage of $\op_{\t}M_{i}(\t)$ under the projection $\cF_{i}\to \op_{\t\in \Sig}\cF_{i}|_{\t^{(m)}(s')}\cong \op_{\t\in \Sig}R_{m}^{\op i}$. 

We let $\wt Y_{m,\Sig}$ be a resolution of singularities of $Y_{m,\Sig}$. Let 
\begin{equation*}
Z:=\Bun^{\l'}_{B,m}(\pi,\Sig)\times \wt Y_{m,\Sig}.
\end{equation*}
Let $\a: Z\to Z'\xr{\a'}\ov\Bun_{B}^{\l}(\pi)$ and $\b: Z\to Z'\xr{\b'}\Bun_{B}^{\l'}(\pi)$ be the compositions.  The diagram \eqref{Z over BunB} is constructed. 

Finally we construct $z\in Z(k)$ such that
\begin{equation}\label{image z}
\a(z)=(s,\cE_{\bu})\in \ov\Bun^{\l}_{B}(\pi)(k), \quad  \b(z)=(s,\cE'_{\bu})\in 
\Bun^{\l'}_{B}(\pi)(k).
\end{equation}
Using the trivialization $i_{\t}$ we get an isomorphism $\t^{(m)}(s)\cong \Spec R_{m}$ for each $\t\in \Sig$. Choose trivializations of $R_{m}$-modules $\io(\t): \cE'_{i}|_{\t^{(m)}(s)}\cong R_{m}^{\op i}$ compatible with the inclusions, for each $\t\in \Sig$. Let 
let $(M_{i}(\t))_{1\le i\le n}\in Y_{m,\t}(k)$ be the chain of $R_{m}$-modules obtained as the image of $\cE_{i}$ under $  \cE_{i}\subset \cE'_{i}\surj \cE'_{i}|_{\t^{(m)}(s)}\xr{\io(\t)} R_{m}^{\op i}$. The tuple $(s,\cE'_{\bu}, \{\io(\t)\}_{\t\in \Sig}, \{M_{\bu}(\t)\}_{\t\in \Sig})$ defines a $k$-point $z'\in Z'(k)$ with the property that  $\a'(z')=(s,\cE_{\bu})$ and $\b'(z')=(s,\cE'_{\bu})$. Let $y\in \wt Y_{m,\Sig}(k)$ be any point lifting $\{M_{\bu}(\t)\}_{\t\in \Sig}\in Y_{m,\Sig}(k)$, so that $z=(s,\cE'_{\bu}, \{\io(\t)\}_{\t\in \Sig}, y)\in Z(k)$ maps to $z'\in Z'(k)$. Therefore \eqref{image z} hold.

\end{proof}

We continue to show that $\cN^{\Eis}_{\pi}$ is closed for $G=\GL_{n}$. In \S\ref{sss:NEis const} we showed that for any open substack $\cU\subset \Bun_{G}(\pi)$ of finite type over $k$, $\cN^{\Eis}_{\pi}\cap T^{*}\cU\subset \cup_{\l\in \L_{\cU}}\cW^{\l}$ for a finite subset $\L_{\cU}\subset \xcoch(T)=\ZZ^{n}$. By Lemma \ref{l:ov W}, we have 
\begin{equation*}
\cN^{\Eis}_{\pi}\cap T^{*}\cU= \bigcup_{\l\in \L_{\cU}}\ov\cW^{\l}\cap T^{*}\cU.
\end{equation*}
Each $\ov\cW^{\l}\cap T^{*}\cU$ is closed in $T^{*}\cU$, $\cN^{\Eis}_{\pi}\cap T^{*}\cU$ is closed in $T^{*}\cU$. This being true for all finite type $\cU$, we conclude that $\cN^{\Eis}_{\pi}$ is closed. This proves the closedness of $\cN^{\Eis}_{\pi}$ for $G=\GL_{n}$.


\sss{Change of groups}\label{sss:NEis change gp} To prove the closedness of $\cN^{\Eis}_{\pi}$ for general $G$, we consider the behavior of $\cN^{\Eis}_{\pi}$ under change of groups. Suppose $f: G\to G'$ is a homomorphism of connected reductive groups with finite kernel, so that $df: \frg\incl \frg'$ is injective. Hence we can write $\frg'=\frg\op V$ where the $G$-module $V$ is the orthogonal complement of $\frg$ under the Killing form of $\frg'$. Dually, we have a $G$-equivariant decomposition
\begin{equation}\label{decomp g*}
\frg'^{*}=\frg^{*}\op V^{*}.
\end{equation}
Let $f_{\Bun}: \Bun_{G}(\pi)\to \Bun_{G'}(\pi)$ be the map induced by $f$. Let $\Pi_{G'}: \Bun_{G'}(\pi)\to S$ be the projection. Then $f_{\Bun}$ induces a map on the cotangent and  relative cotangent bundles
\begin{eqnarray*}
&&df_{\Bun}: T^{*}\Bun_{G'}(\pi)|_{\Bun_{G}(\pi)}\to T^{*}\Bun_{G}(\pi),\\
&&df_{\Bun,\rel}: T^{*}(\Pi_{G'})|_{\Bun_{G}(\pi)}\to T^*(\Pi).
\end{eqnarray*}
We claim that $df_{\Bun,\rel}$ has a canonical section. Indeed, the fiber of $T^{*}(\Pi_{G'})|_{\Bun_{G}(\pi)}$ at $(s,\cE)\in \Bun_{G}(\pi)$ is $\cohog{0}{\frX_{s}, \cE(\frg'^{*})\ot \om_{\frX_{s}}}$, which under the decomposition \eqref{decomp g*} decomposes as
\begin{equation*}
\cohog{0}{\frX_{s}, \cE(\frg^{*})\ot \om_{\frX_{s}}}\op \cohog{0}{\frX_{s}, \cE(V^{*})\ot \om_{\frX_{s}}}.
\end{equation*}
The inclusion of the first summand gives a section $\g_{\rel}$ to $df_{\Bun,\rel}$ whose image we denote by $\G_{\rel}$. We can then define a conic substack $\G\subset T^{*}\Bun_{G'}(\pi)|_{\Bun_{G}(\pi)}$ as the preimage of $\G_{\rel}$ under the projection $ T^{*}\Bun_{G'}(\pi)|_{\Bun_{G}(\pi)}\to T^*(\Pi_{G'})|_{\Bun_{G}(\pi)}$. Then the composition
\begin{equation*}
\G\subset T^{*}\Bun_{G'}(\pi)|_{\Bun_{G}(\pi)}\to T^{*}\Bun_{G}(\pi)
\end{equation*}
is an isomorphism and gives a section $\g$ to $df_{\Bun}$ compatible with $\g_{\rel}$. We therefore get a (counter-intuitive) map
\begin{equation}\label{wrong way}
\r: T^{*}\Bun_{G}(\pi)\xr{\g}T^{*}\Bun_{G'}(\pi)|_{\Bun_{G}(\pi)}\xr{f_{\Bun}^{\na}}T^{*}\Bun_{G'}(\pi).
\end{equation}

\begin{lemma}\label{l:Eis GG'} 
Under the above notations, we have an equality of subsets of $T^{*}\Bun_{G}(\pi)$
\begin{equation*}
\cN^{\Eis}_{\pi}=\r^{-1}(\cN^{\Eis}_{\pi,G'}).
\end{equation*}
\end{lemma}
\begin{proof}
Let us choose the Borel subgroups $B\subset G$ and $B'\subset G'$ compatible with the decomposition \eqref{decomp g*} in the following way. Let $\l: \Gm\to  G$ be a generic homomorphism such that the non-negative weights of $\Ad(\l)$ on $\frg$ defines a Borel subalgebra $\frb$, hence a Borel subgroup $B$ (the attracting locus to $1$ under $\l$). Let $T=C_{G}(\l)$, a maximal torus of $B$.  Let $V_{0}$ be the $0$-weight space of $\Ad(\l)$ on $V$ and  $V_{\ge0}$ be the direct sum of weight spaces with weights $\ge0$. Then $\frb\op V_{\ge0}$ is the Lie algebra of the parabolic subgroup $P'\subset G'$, the attracting locus to $1\in G'$ under $f\c\l$. Then $L'=C_{G'}(\l)$ is a Levi subgroup of $P'$ with Lie algebra $\frl'=\frt\op V_{0}$. Let $B'_{1}\subset L'$ be any Borel subgroup containing $T$, and let $B'=N_{P'}B'_{1}\subset P'$ be the corresponding Borel subgroup of $P'$, hence of $G'$. By construction, on the level of Lie algebras we have a $B$-equivariant decomposition $\frb'=\frb\op (V\cap \frb')$. In particular, we have a $B$-equivariant decomposition 
\begin{equation}\label{decomp bper}
\frb'^{\bot}=\frb^{\bot}\op (V\cap \frb')^{\bot}
\end{equation}
where $\frb^{\bot}\subset \frg^{*}$ (resp. $ (V\cap \frb')^{\bot}\subset V^{*}$) is the orthogonal complement of $\frb$ (resp. $(V\cap \frb')^{\bot}$).

Let $\L_{B}\subset T^{*}\Bun_{G}(\pi)|_{\Bun_{B}(\pi)}$ be the preimage of the zero section under its projection $d\frp$ to $T^{*}\Bun_{B}(\pi)$; similarly define $\L_{B'}\subset T^{*}\Bun_{G'}(\pi)|_{\Bun_{B'}(\pi)}$. Let $\L^{\rel}_{B}\subset T^*(\Pi)|_{\Bun_{B}(\pi)}$ and $\L_{B'}^{\rel}\subset T^{*}(\Pi_{G'})|_{\Bun_{B'}(\pi)}$ be the images of $\L_{B}$ and $\L_{B'}$ in  the relative cotangent bundles. Note that the projections $\L_{B}\to \L_{B}^{\rel}$ and $\L_{B'}\to \L_{B'}^{\rel}$ are bijective on points. The differential of the map $\b:\Bun_{B}(\pi)\to \Bun_{B'}(\pi)$ then induces maps
\begin{equation*}
\xymatrix{ \L_{B'}|_{\Bun_{B}(\pi)}\ar[d]^{\wr}\ar[r]^-{d\b}& \L_{B}\ar[d]^{\wr}\\
\L_{B'}^{\rel}|_{\Bun_{B}(\pi)}\ar[r]^-{d\b_{\rel}}& \L_{B}^{\rel}}
\end{equation*}
The decomposition \eqref{decomp bper} gives a section $t_{\rel}$ to $d\b_{\rel}$ (similar to the construction of $\g_{\rel}$), hence a section $t$ of $d\b$. We then have a map
\begin{equation*}
\t: \L_{B}\xr{t}\L_{B'}|_{\Bun_{B}(\pi)}\to \L_{B'}
\end{equation*}
that fits into a commutative diagram
\begin{equation*}
\xymatrix{\L_{B} \ar[d]^{\frq} \ar[r]^-{\t}& \L_{B'}\ar[d]^{\frq'}\\
T^{*}\Bun_{G}(\pi)\ar[r]^-{\r} & T^{*}\Bun_{G'}(\pi)}
\end{equation*}
Here $\frq$ and $\frq'$ are the restrictions of $\frp^{\na}$ and $\frp'^{\na}$. We therefore have a tautological  inclusion
\begin{equation*}
\cN^{\Eis}_{\pi}=\frq(\L_{B})\subset \r^{-1}(\frq'(\L_{B'}))=\r^{-1}(\cN^{\Eis}_{\pi, G'}).
\end{equation*}
To show this is an equality, observing that both sides above project injectively to the relative cotangent $T^*(\Pi)$ (which has been shown in the first paragraph of the proof of Theorem \ref{th:Eis cone closed}), it suffices to check their images in $T^*(\Pi)$ are the same. This allows us to reduce to the case where $S$ is a point. In this case, we need to show that $\cN_{X}=\r^{-1}(\cN_{X,G'})$, i.e., a $G$-Higgs bundle $(\cE,\ph)$ is nilpotent if and only if the $G'$-Higgs bundle $(\cE'=\cE\times^{G}G', (\ph,0))$, where $(\ph,0)$ is a section of $\cE'(\frg'^{*})\ot\om=\cE(\frg^{*})\ot\om\op \cE(V^{*})\ot\om$ by the decomposition \eqref{decomp g*}, is nilpotent. This last claim follows from the obvious relationship between the nilpotent cones $\cN_{\frg^{*}}=\frg^{*}\cap \cN_{\frg'^{*}}$. The proof of the lemma is finished.
\end{proof}

\sss{Closedness of $\cN^{\Eis}_{\pi}$: general case}\label{sss:NEis closed gen}
Choose an embedding $f: G\incl G'=\GL_{n}$ for some $n$. By Lemma \ref{l:Eis GG'}, $\cN^{\Eis}_{\pi}$ is the preimage of $\cN^{\Eis}_{\pi,G'}$ under the wrong way map \eqref{wrong way}. We have shown in \S\ref{sss:NEis closed GLn} that $\cN^{\Eis}_{\pi,G'}$ is closed in $T^{*}\Bun_{G'}(\pi)$, therefore $\cN^{\Eis}_{\pi}$ is closed in $T^{*}\Bun_{G}(\pi)$. 

We now finish the proof of Theorem \ref{th:Eis cone closed}. The only remaining statement is that $P^{\Eis}_\pi=P_\pi|_{\cN^{\Eis}_\pi}: \cN^{\Eis}_\pi\to \cN^{\rel}_\pi$ is a universal homeomorphism. Since $P_\pi: T^*\Bun_G(\pi)\to T^*\Pi$ is affine and $\cN^{\Eis}_\pi$ is closed in $T^*\Bun_G(\pi)$, $P^{\Eis}_\pi$ is also affine. By Lemma \ref{l:NEis bij Nrel}, $P^{\Eis}_\pi$ is a bijection on geometric points. Therefore $P^{\Eis}_\pi$ is finite, hence universally closed. Using that $P^{\Eis}_{\pi}$ is bijective on geometric points again, we conclude that it is a universal homeomorphism. This finishes the proof of Theorem \ref{th:Eis cone closed}.

\subsection{Universal Eisenstein cone with level structure}\label{ss:Eis cone level}
In this subsection we extend the construction of the universal Eisenstein cone to the moduli of bundles with level structures. 

Recall that $\Pi^{N}_{\s}: \Bun_{G,N}(\pi,\s)\to S$ and its relative cotangent stack $T^{*}(\Pi^{N}_{\s})$. Similarly we have the $B$-level structure versions $\Pi^{B}_{\s}$ and $T^{*}(\Pi^{B}_{\s})$.

We make a family version of the constructions in \S\ref{sss:univ cone B}. Let $\Bun_{G,1}(\pi, \s)$ (resp. $\Bun_{B,1}(\pi, \s)$) be the moduli stacks of $G$-bundles (resp. $B$-bundles) over fibers of $\pi$ together with a trivialization along $\s$. Let 
$\Om_{G,\pi,\s}$  (resp. $\Om_{B,\pi,\s}$) denote the descent of the cotangent bundles of $\Bun_{G,1}(\pi, \s)$ (resp. $\Bun_{B,1}(\pi, \s)$) to $\Bun_{G}(\pi)$ (resp. $\Bun_{B}(\pi)$).  
Let $\frp_{\s}: \Bun_{B,1}(\pi, \s)\to \Bun_{G,1}(\pi, \s)$ be the induction map, giving rise to the correspondence
\begin{equation*}
\xymatrix{ \Om_{B,\pi,\s} & \Bun_{B}(\pi)\times_{\Bun_{G}(\pi)}\Om_{G,\pi,\s}\ar[l]_-{d\frp_{\s}}\ar[r]^-{\frp^{\na}_{\s}}  & \Om_{G,\pi,\s} 
}
\end{equation*}
Define
\begin{equation*}
\un\cN^{\Eis}_{\pi,\s}=\frp^{\na}_{\s}((d\frp_{\s}^{-1}(0_{\Bun_{B}(\pi)})))\subset \Om_{G,\pi,\s}.
\end{equation*}
We have the natural maps
\begin{equation*}
\xymatrix{
\frq: T^{*}\Bun_{G,B}(\pi, \sigma)\ar[r] &  \Om_{G, \pi,\sigma}
&
\frq': T^{*}\Bun_{G,N}(\pi, \sigma)\ar[r] &  \Om_{G, \pi,\sigma}
}\end{equation*}
 given by the differentials of the respective projections
$\Bun_{G,1}(\pi,\sigma)\to \Bun_{G,B}(\pi,\sigma)$
and $\Bun_{G,1}(\pi,\sigma)\to \Bun_{G,N}(\pi,\sigma)$.

\begin{defn}\label{def:univ cone Eis}
\begin{enumerate}
\item The universal Eisenstein cone $\cN^{\Eis, B}_{\pi, \s}\subset T^{*}\Bun_{G,B}(\pi,\s)$ is defined to be $\frq^{-1}(\un\cN^{\Eis}_{\pi,\s})$.
\item The universal Eisenstein cone $\cN^{\Eis, N}_{\pi, \s}\subset T^{*}\Bun_{G,N}(\pi,\s)$ 
is defined to be $(\frq')^{-1}(\un\cN^{\Eis}_{\pi,\s})$.\end{enumerate}
\end{defn}

\begin{remark}\label{rem: cone B vs N}
It is elementary to check: for the natural $R_{\s/S}(H)$-torsor $q:\Bun_{G,N}(\pi,\sigma)\to \Bun_{G,B}(\pi,\sigma)$
and the $R_{\s/S}(G/B)$-fibration $r:\Bun_{G,B}(\pi,\sigma)\to \Bun_{G}(\pi)$, see \eqref{forg level}, we have
\begin{equation*}
\oll{q}(\cN^{\Eis,B}_{\pi,\s}) = \cN^{\Eis,N}_{\pi,\s},
\quad \orr{q}(\cN^{\Eis,N}_{\pi,\s}) = \cN^{\Eis,B}_{\pi,\s},
\quad \orr{r}(\cN^{\Eis,B}_{\pi,\s}) = \cN^{\Eis}_{\pi}.
\end{equation*}
\end{remark}

The main result of this subsection is: 
\begin{theorem}\label{th:univ cone Iw} Let $G$ be any reductive group over $k$.
\begin{enumerate}
    \item The universal Eisenstein cone $\cN^{\Eis,B}_{\pi,\s}$ is a closed conic Lagrangian inside $T^{*}\Bun_{G,B}(\pi,\s)$. 
    \item The projection $P^{B}_{\pi,\s}$ restricts to a universal homeomorphism $P^{\Eis,B}_{\pi,\s}: \cN^{\Eis,B}_{\pi,\s}\to \cN^{\rel,B}_{\pi,\s}$. 
    \item In particular, $\cN^{\Eis,B}_{\pi,\s}$ is the unique lifting of the relative nilpotent cone $\cN^{\rel,B}_{\pi,\s}$ as in Proposition \ref{p:lifting Lag}.
\end{enumerate}
Analogous statements are true for $\cN^{\Eis,N}_{\pi,\s}$.
\end{theorem}

Again we first record some consequences.
Similar to Corollary \ref{c:univ cone non-char}, we have: 
\begin{cor} The universal Eisenstein cones $\cN^{\Eis,B}_{\pi,\s}$ and $\cN^{\Eis,N}_{\pi,\s}$ are non-characteristic with respect to the projections $\Pi^{B}_{\s}:\Bun_{G,B}(\pi,\s)\to S$ and $\Pi^{N}_{\s}:\Bun_{G,N}(\pi,\s)\to S$ respectively. 
\end{cor}

\sss{Compatibility with base change}
Let $S'$ be a stack smooth over $k$, and $t: S'\to S$ be a map of stacks. Let $\pi':\frX'=\frX\times_{S}S'\to S'$ be the base change of $\pi$, and $\s'=\s\times_{S}S'$. Note the natural identification
$\Bun_{G,B}(\pi,\s)\times_{S}S'\simeq\Bun_{G,B}(\pi',\s')$. Let
\begin{equation*}
\xymatrix{
\th_{G,B}: \Bun_{G,B}(\pi',\s')\ar[r] &  \Bun_{G,B}(\pi,\s)
}
\end{equation*}
be the natural projection. Passing to cotangent bundles we have the map
\begin{equation*}
\xymatrix{
d\th_{G,B}: \Bun_{G,B}(\pi',\s')\times_{\Bun_{G,B}(\pi,\s)}T^{*}\Bun_{G,B}(\pi,\s)\cong S'\times_{S} T^{*}\Bun_{G,B}(\pi,\s) \ar[r] &  T^{*}\Bun_{G,B}(\pi',\s').
}
\end{equation*}

We have the following extension of Lemma \ref{l:Eis cone bc}.

\begin{prop}\label{p:Eis cone level bc} The map $d\th_{G,B}$ restricts to a universal homeomorphism $S'\times_{S}\cN^{\Eis,B}_{\pi,\s}\to \cN^{\Eis,B}_{\pi',\s'}$. In particular, we have $\cN^{\Eis,B}_{\pi',\s'}=\oll{\th_{G,B}}(\cN^{\Eis,B}_{\pi,\s})$.

Same statements hold when $B$-level structure is replaced by $N$-level structure. 
\end{prop}
\begin{proof} Let $d\un\th_{1}: S'\times_{S}\Om_{G,\pi,\s}\to \Om_{G,\pi',\s'}$ be the map induced by the differential of $\th_{1}: \Bun_{G,1}(\pi', \s')\to \Bun_{G,1}(\pi,\s)$. The same argument of Lemma \ref{l:Eis cone bc} shows that $d\un\th_{1}$ restricts to a bijection
\begin{equation}\label{bc un cN}
S'\times_{S}\un\cN^{\Eis}_{\pi,\s}\isom \un\cN^{\Eis}_{\pi',\s'}.
\end{equation}
We have a Cartesian diagram
\begin{equation}\label{Cart th}
\xymatrix{ S'\times_{S}T^{*}\Bun_{G,B}(\pi,\s)\ar[d]_{\id_{S'}\times\frq} \ar[r]^-{d\th_{G,B}} & T^{*}\Bun_{G,B}(\pi',\s')\ar[d]^{\frq'} \\
S'\times_{S}\Om_{G,\pi,\s} \ar[r]^-{d\un\th_{1}} & \Om_{G,\pi',\s'}}
\end{equation}
By the definition of $\cN^{\Eis,B}_{\pi,\s}$, we get
\begin{equation*}
S'\times_{S}\cN^{\Eis,B}_{\pi,\s}=(\id_{S'}\times\frq)^{-1}(S'\times_{S}\un\cN^{\Eis}_{\pi,\s}), \quad \cN^{\Eis,B}_{\pi',\s'}=(\frq')^{-1}(\un\cN^{\Eis}_{\pi',\s'}).
\end{equation*}
Therefore the desired bijection follows from \eqref{bc un cN} and the Cartesian diagram \eqref{Cart th}.

The statements about $N$-level structure is proved similarly. 
\end{proof}

Now we begin the proof of Theorem \ref{th:univ cone Iw}, following the same lines of argument for Theorem \ref{th:Eis cone closed}. We will give the proof for $\cN^{\Eis,B}_{\pi,\s}$ only; the case of $\cN^{\Eis,N}_{\pi,\s}$ can be proved similarly or deduced from Remark~\ref{rem: cone B vs N}.

\sss{$P^{B}_{\pi,\s}$ induces a bijection $\cN^{\Eis,B}_{\pi,\s}\to\cN^{\rel,B}_{\pi,\s}$} Let $\Om^{\rel}_{G,\pi,\s}\to \Bun_{G}(\pi,\s)$ be the descent of the relative cotangent $T^{*}(\Pi^{1}_{\s})$ (where $\Pi^{1}_{\s}: \Bun_{G,1}(\pi,\s)\to S$) to $\Bun_{G}(\pi,\s)$. We have a Cartesian diagram
\begin{equation*}
\xymatrix{T^{*}\Bun_{G,B}(\pi,\s)\ar[r]^-{\frq}\ar[d]_{P^{B}_{\pi,\s}} & \Om_{G,\pi,\s}\ar[d]^{\un P_{\pi,\s}}\\
T^{*}\Pi^{B}_{\s} \ar[r]^-{\frq^{\rel}} & \Om^{\rel}_{G,\pi,\s}}
\end{equation*}
Let $\un\cN^{\rel}_{\pi,\s}\subset \Om^{\rel}_{G,\pi,\s}$ be the zero fiber of the relative version of the Hitchin map $\Bun_{G,1}(\pi,\s)$. Then $\cN^{\rel,B}_{\pi,\s}=\frq^{\rel,-1}(\un\cN^{\rel}_{\pi,\s})$ as can be checked fiberwise over $S$.  Since $\cN^{\Eis,B}_{\pi,\s}=\frq^{-1}(\un\cN^{\Eis}_{\pi,\s})$, to show $P^{B}_{\pi,\s}$ maps $\cN^{\Eis,B}_{\pi,\s}$ to $\cN^{\rel, B}_{\pi,\s}$ and restricts to a bijection on geometric points, it suffices to show:

\begin{lemma}\label{l:Om bij rel}
The map $\un P_{\pi,\s}$ maps $\un\cN^{\Eis}_{\pi,\s}$ to $\un\cN^{\rel}_{\pi,\s}$, and restricts to a bijection on geometric points.
\end{lemma}
\begin{proof}[Proof sketch]
The argument is same as Lemma \ref{l:NEis bij Nrel}. The non-obvious part is the injectivity of $\un\cN^{\Eis}_{\pi,\s}(k)\to \un\cN^{\rel}_{\pi,\s}(k)$. If $(s,\cE, \wt\ph)$ and $(s,\cE, \wt\ph')\in \un\cN^{\Eis}_{\pi,\s}(k)$ both map to the same $(s,\cE,\ph)\in \un\cN^{\rel}_{\pi,\s}(k)$, then there are Borel reductions $\cE_{B}$ and $\cE'_{B}$ of $\cE$ such that $\wt\ph\in\ker(\Om_{G,\pi,\s}|_{(s,\cE)}\to \Om_{B,\pi,\s}|_{(s,\cE_{B})})$, and $\wt\ph'\in\ker(\Om_{G,\pi,\s}|_{(s,\cE)}\to \Om_{B,\pi,\s}|_{(s,\cE'_{B})})$. We can then connect $\cE_{B}$ and $\cE'_{B}$ by a chain of Borel reductions $\cE^{(i)}_{B}$ of $\cE$ with relative position a simple reflection for neighboring reductions, such that $\ph\in \cohog{0}{\frX_{s},\cE^{(i)}_{B}(\frn)\ot\om_{\frX_{s}}}$ for all $i$. One then shows inductively that $\wt\ph\in\ker(\Om_{G,\pi,\s}|_{(s,\cE)}\to \Om_{B,\pi,\s}|_{(s,\cE^{(i)}_{B})})$, hence eventually $\wt\ph\in\ker(\Om_{G,\pi,\s}|_{(s,\cE)}\to \Om_{B,\pi,\s}|_{(s,\cE'_{B})})$. Thus $T^{*}_{s}S\ni \wt\ph-\wt\ph'\in \ker(\Om_{G,\pi,\s}|_{(s,\cE)}\to \Om_{B,\pi,\s}|_{(s,\cE'_{B})})$, forcing $\wt\ph-\wt\ph'$ to be zero.
\end{proof}



\sss{$\cN^{\Eis,B}_{\pi,\s}$ is a Lagrangian}\label{sss:NEis level Lag}
To show that  $\cN^{\Eis,B}_{\pi,\s}$ is a Lagrangian, we may make an \'etale base change $\th: S'\to S$ and reduced to show that $\cN^{\Eis,B}_{\pi',\s'}$ is a Lagrangian. In particular, we may assume that $\s\to S$ is a trivial covering, i.e., $\s\cong S\times \Sigma$ for a finite set $\Sig$.

We first show $\cN^{\Eis,B}_{\pi,\s}$ is constructible and isotropic.  Similarly to Lemma \ref{l:nilcone B alt}, we have maps
\begin{equation}\label{beta}
\xymatrix{(B\bs G/B)^{\Sig} & \Bun_{B,1}(\pi,\s)\twtimes{B^{\Sig}}(G/B)^{\Sig}\cong \Bun_{B}(\pi)\times_{\Bun_{G}(\pi)}\Bun_{G,B}(\pi,\s) \ar[l]_-{\b}\ar[r]^-{\a} & \Bun_{G,B}(\pi,\s)}
\end{equation}
and an equality
\begin{equation}\label{NEis transport BGB}
\cN^{\Eis,B}_{\pi,\s}=\orr{\a}\oll{\b}(T^{*}(B\bs G/B)^{\Sig}).
\end{equation}
We can also restrict \eqref{beta} to the $\l$-component of $\Bun_{B,1}(\pi,\s)$ (where the restrictions of $\a$ and $\b$ are denoted $\a^{\l}$ and $\b^{\l}$), and set
\begin{equation}\label{cW level}
\cW^{\l}_{\s}:=\orr{\a^{\l}}\oll{\b^{\l}}(T^{*}(B\bs G/B)^{\Sig}).
\end{equation}
Lemma \ref{l:pres iso} shows that $\cW^{\l}_{\s}$ is isotropic. For any finite type open substack $\cU\subset \Bun_{G,B}(\pi,\s)$, we argue as in \S\ref{sss:NEis const} that finitely many $\cW^{\l}_{\s}$ cover $\cN^{\Eis,B}_{\pi,\s}\cap T^{*}\cU$, which implies that $\cN^{\Eis,B}_{\pi,\s}\cap T^{*}\cU$ is constructible and isotropic. Hence $\cN^{\Eis,B}_{\pi,\s}$ is constructible and isotropic.

By constructibility, the pointwise bijection between $\cN^{\Eis,B}_{\pi,\s}$ and $\cN^{\rel, B}_{\pi,\s}$ implies $\dim \cN^{\Eis,B}_{\pi,\s}=\dim \cN^{\rel, B}_{\pi,\s}$, and the latter has the same dimension as $\Bun_{G,B}(\pi,\s)$ by Lemma \ref{l:rel nilp cone}, we conclude that $\cN^{\Eis,B}_{\pi,\s}$ is a Lagrangian in $T^{*}\Bun_{G,B}(\pi,\s)$.

\sss{$\cN^{\Eis,B}_{\pi,\s}$ is closed} 
It suffices to show that $\un\cN^{\Eis}_{\pi,\s}$ is closed in $\Om_{G,\pi,\s}$. For this, we proceed as in \S\ref{sss:NEis change gp}--\S\ref{sss:NEis closed gen} to reduce to the case of $\GL_{n}$. Below we assume $G=\GL_{n}$. Making an \'etale base change we assume the cover $\s\to S$ is trivial, i.e., $\s$ is a disjoint union of sections $\t: S\to X$ indexed by $\t$ in some finite set $\Sig$. 

We consider a version of $\ov\Bun^{\l}_{B}(\pi)$ with $B$-level structures along $\s$. For a scheme $S'$ and a rank $n$ vector bundle $\cV$ over $S'$, a {\em lax flag} of $\cV$ is a chain of vector bundles together with maps over $S'$
\begin{equation*}
\cV_{1}\to \cV_{2}\to \cdots \cV_{n-1}\to \cV_{n}=\cV
\end{equation*}
such that $\cV_{i}$ has rank $i$. The maps are not required to be injective. Let $\fl^{\lax}_{n}$ be the moduli space over $k$ whose $S'$-points are the set of lax flags of the trivial vector bundle $\cO_{S'}^{\op n}$. Then $G$ acts on $\fl^{\lax}_{n}$ by acting on the trivial bundle. Also $\fl^{\lax}_{n}$ contains the flag variety $\fl_{n}$ of $G$ as an open subscheme.

Let $\ov\cB^{\l}_{\s}=\ov\Bun^{\l}_{B}(\pi)\times_{\Bun_{G}(\pi)}\Bun_{G,B}(\pi,\s)$, which contains $\cB^{\l}_{\s}=\Bun^{\l}_{B}(\pi)\times_{\Bun_{G}(\pi)}\Bun_{G,B}(\pi,\s)$ as an open substack. Since $\ov\Bun^{\l}_{B}(\pi)$ is smooth over $S$, and $\Bun_{G,B}(\pi,\s)\to \Bun_G(\pi)$ is smooth, $\ov\cB^{\l}_{\s}$ is also smooth over $S$.

Restricting a point of $\ov\Bun^{\l}_{B}(\pi)$ along a section $\t\in \Sig$ gives a lax flag of the underlying rank $n$ bundle, hence a canonical map $\ev_{\t}: \ov\Bun^{\l}_{B}(\pi)\to G\bs \fl^{\lax}_{n}$. Together with the information of the extra $B$-reduction along $\t$, we get a map
\begin{equation*}
\ov\b^{\l}_{\t}: \ov\cB^{\l}_{\s}\to G\bs (\fl^{\lax}_{n}\times \fl_{n}).
\end{equation*}
Taking product over all $\t\in \Sig$ we get a map $\ov\b^{\l}_{\s}=\prod_{\t\in \Sig}\ov\b^{\l}_{\t}: \ov\cB^{\l}_{\s}\to (G\bs (\fl^{\lax}_{n}\times \fl_{n}))^{\Sig}$. Together with the projection to $\Bun_{G,B}(\pi,\s)$ we have a diagram
\begin{equation*}
\xymatrix{(G\bs (\fl^{\lax}_{n}\times \fl_{n}))^{\Sig} &\ar[l]_-{\ov\b^{\l}_{\s}} \ov\cB^{\l}_{\s}\ar[r]^-{\ov\a^{\l}_{\s}} &  \Bun_{G,B}(\pi,\s). }
\end{equation*}
Let
\begin{equation*}
\ov\cW^{\l}_{\s}:=\orr{\ov\a^{\l}_{\s}}\oll{\ov\b^{\l}_{\s}}(T^{*}(G\bs (\fl^{\lax}_{n}\times \fl_{n}))^{\Sig})\subset T^{*}\Bun_{G,B}(\pi,\s).
\end{equation*}
Since $\ov\a^{\l}_{\s}$ is proper (being the base change of $\ov\frp^{\l}:\ov\Bun^{\l}_{B}(\pi)\to \Bun_{G}(\pi)$), $\ov\cW^{\l}_{\s}$ is closed in $T^{*}\Bun_{G,B}(\pi,\s)$.  Once we prove the following lemma, which is an analog of Lemma \ref{l:ov W}, the rest of the argument is the same as in the case without level structures. 


\begin{lemma}\label{l:ov W level}
We have
\begin{equation*}
\cN^{\Eis, B}_{\pi,\s}=\bigcup_{\l\in \ZZ^{n}}\ov\cW^{\l}_{\s}.
\end{equation*}
\end{lemma}
\begin{proof} Recall $\cW^{\l}_{\s}$ from \eqref{cW level}. By \eqref{NEis transport BGB}, $\cN^{\Eis, B}_{\pi,\s}=\cup_{\l\in \ZZ^{n}}\cW^{\l}_{\s}$. Clearly $\cW^{\l}_{\s}\subset\ov\cW^{\l}_{\s}$, hence $\cN^{\Eis, B}_{\pi,\s}\subset \cup_{\l\in \ZZ^{n}}\ov\cW^{\l}_{\s}$. To show the other inclusion, it suffices to show for each $\l\in \ZZ^{n}$ that
\begin{equation}\label{ov W level estimate}
\ov\cW^{\l}_{\s}\subset \bigcup_{\l'\ge \l}\cW^{\l'}_{\s}.
\end{equation}
For a point $(s,\cE, \{\cF(\t)\}_{\t\in \Sig}, \xi)\in \ov\cW^{\l}_{\s}(k)$ (where $\cF(\t)$ is a full flag of $\cE|_{\t(s)}$ for each $\t\in \Sig$), there exists a chain of subsheaves $\cE_{\bu}=(\cE_{i})_{1\le i\le n}$ of $\cE$ (with $\rk \cE_{i}=i$, $\cE_{n}=\cE$) such that the image of $\xi$ in $T^{*}_{(s,\cE_{\bu}, \{\cF(\t)\})}\ov\cB^{\l}_{\s}$ lies in the image of pullback from $T^{*}(G\bs (\fl^{\lax}_{n}\times \fl_{n}))^{\Sig}$ via $d\ov\b^{\l}_{\s}$. We will use the notations from the proof of Lemma \ref{l:ov W}: the saturation $\cE'_{i}$ of $\cE_{i}$, $\l'=(\l'_{1},\cdots, \l'_n)$, etc. We will construct a diagram
\begin{equation}\label{diag Z level}
\xymatrix{ & (G\bs (\fl^{\lax}_{n}\times \fl_{n}))^{\Sig} & \ar[l]_-{\ov\b^{\l}_{\s}} \ov\cB^{\l}_{\s} \ar[dr]^{\ov\a^{\l}_{\s}} \\
Y^{B}_{\s}\ar[ur]^{\y}\ar[dr]^{\th} & \ar[l]_-{b} Z^{B}_{\s}\ar[rr]^-{\z} \ar[ur]^{\g}\ar[dr]^{\d} & & \Bun_{G,B}(\pi,\s)\\
 & (B\bs \fl_{n})^{\Sig} &\ar[l]^-{\b^{\l'}_{\s}} \cB^{\l'}_{\s}\ar[ur]^{\a^{\l'}_{\s}} }
\end{equation}
such that 
\begin{itemize}
\item The map $(b,\d): Z^{B}_{\s}\to Y^{B}_{\s}\times_{(B\bs \fl_{n})^{\Sig}}\cB^{\l'}_{\s}$ is smooth.
\item There is $z\in Z^{B}_{\s}(k)$ with image $(s,\cE_{\bu}, \{\cF(\t)\}_{\t\in \Sig})\in \ov\cB^{\l}_{\s}(k)$ and image $(s,\cE'_{\bu}, \{\cF(\t)\}_{\t\in \Sig})\in \cB^{\l'}_{\s}(k)$.
\end{itemize}
If we have such a diagram, let $y\in Y^{B}_{\s}(k)$ be the image of $z$, then by assumption, $d\ov\a^{\l}_{\s}(\xi)\in \Im(d\ov\b^{\l}_{\s}|_{\y(y)})$. This implies $d\z(\xi)\in \Im(db|_{y})$. On the other hand, $d\z(\xi)=d\d(d\a^{\l'}_{\s}(\xi))$. Let $\xi'=d\a^{\l'}_{\s}(\xi)\in T^{*}_{(s,\cE'_{\bu}, \{\cF(\t)\})}\cB^{\l}_{\s}$. We have $d\d(\xi')\in \Im(db|_{y})$. The smoothness of $(b,\d):Z^{B}_{\s}\to Y^{B}_{\s}\times_{(B\bs \fl_{n})^{\Sig}}\cB^{\l'}_{\s}$ implies that the following is exact
\begin{equation*}
0\to T^{*}_{\th(y)}(B\bs \fl_{n})^{\Sig}\xr{(d\th, d\b^{\l'}_{\s})} T^{*}_{y}Y^{B}_{\s}\op T^{*}_{(s,\cE'_{\bu}, \{\cF(\t)\})}\cB^{\l}_{\s}\xr{db-d\d} T^{*}_{z}Z^{B}_{\s}
\end{equation*}
The fact that $d\d(\xi')\in \Im(db|_{y})$ then implies that $\xi'\in \Im(d\b^{\l'}_{\s})$, i.e., $(s,\cE,\{\cF(\t)\}_{\t\in \Sig}, \xi)\in \cW^{\l'}_{\s}$.

It remains to construct the diagram \eqref{diag Z level} with the desired properties. Again we use notation from the proof of Lemma \ref{l:ov W}, except that we denote the finite set $\Sig$ there (indexing a set of sections of $\pi$) now by $\Sig'$. Enlarging $\Sig'$ if necessary we may assume $\Sig\subset \Sig'$. Recall $m$ is the maximum of the multiplicities of any single point appearing in the divisors $D_{i}$, $1\le i\le n-1$. Consider the scheme $Y_{m+1,\Sig'}=\prod_{\t\in \Sig'}Y_{m+1,\t}$. Note that for each $\t\in \Sig$, a $k$-point $M_{i}(\t)\subset \cO_{R_{m+1}}^{\op i}$ of $Y_{m+1,\t}$ after reduction modulo the maximal ideal of $R_{m+1}$ gives an $i$-dimensional vector space $V_{i}(\t)=M_{i}(\t)\ot_{R_{m+1}}k$ together with maps $V_{1}\to V_{2}\to\cdots  V_{n-1}\to V_{n}=k^{n}$. Here $\dim V_{i}=i$ is guaranteed by the fact that $m+1$ is strictly larger than the length of $\cO_{R_{m+1}}^{\op i}/M_{i}(\t)$. This construction gives a map $Y_{m+1,\t}\to \fl^{\lax}_{n}$. Taking product over $\t\in \Sig$ we get a map $\y': Y_{m+1,\Sig'}\to (\fl^{\lax}_{n})^{\Sig}$. Let $\wt Y_{m+1,\Sig'}$ be a resolution of singularities of $Y_{m+1,\Sig'}$ and let $Y^{B}_{\s}:=\wt Y_{m+1,\Sig'}\times \fl_{n}^{\Sig}$. The map $\y: Y^{B}_{\s}\to (G\bs (\fl^{\lax}_{n}\times \fl_{n}))^{\Sig}$ is the composition
\begin{equation*}
Y^{B}_{\s}\to Y_{m+1,\Sig'}\times \fl_{n}^{\Sig}\xr{\y'\times\id}(\fl^{\lax}_{n})^{\Sig}\times \fl_{n}^{\Sig}\to (G\bs (\fl^{\lax}_{n}\times \fl_{n}))^{\Sig}.
\end{equation*}
The map $\th: Y^{B}_{\s}\to (B\bs \fl_{n})^{\Sig}$ is the composition of the second projection and the further projection $\fl_{n}^{\Sig}\to (B\bs \fl_{n})^{\Sig}$. 

Let $Z^{B}_{\s}=\Bun^{\l'}_{B,m+1}(\pi,\Sig')\times Y^{B}_{\s}$, with the projection to $Y^{B}_{\s}$ denoted by $b$. Let $\d$ be the composition
\begin{equation*}
Z^{B}_{\s}=\Bun^{\l'}_{B,m+1}(\pi,\Sig')\times Y^{B}_{\s}\to \Bun^{\l'}_{B,1}(\pi,\s)\times \fl_{n}^{\Sig}\to \Bun_{B,1}^{\l'}(\pi,\s)\twtimes{B^{\Sig}}\fl_{n}^{\Sig}=\cB^{\l'}_{\s}.
\end{equation*}
Recall the map $\a: \Bun^{\l'}_{B,m+1}(\pi,\Sig')\times Y_{m+1,\Sig'}\to \ov\Bun^{\l}_{B}(\pi)$ constructed in the proof of Lemma \ref{l:ov W} that sends $(s',\cF_{\bu}, \{\io(\t)\}_{\t\in \Sig'}, \{M_{\bu}(\t)\}_{\t\in \Sig'})$ to the point $(s', \cG_{\bu})\in \ov\Bun^{\l}_{B}(\pi)$ obtained by modifying $\cF_{\bu}$ along $\Sig'(s')$. Note that $\cG_{n}=\cF_{n}$ is equipped with a trivialization along $\Sig'(s')$ using $\{\io(\t)\}_{\t\in \Sig'}$, and in particular along $\Sig(s')$. With the extra factor of $\fl_{n}^{\Sig}$ in $Z^{B}_{\s}$, we can assign a full flag for $\cG_{n}|_{\t(s')}$ for each $\t\in \Sig$. This way we get a map $\g: Z^{B}_{\s}\to \ov\cB^{\l}_{\s}$. We have constructed all spaces and maps in the diagram \eqref{diag Z level}. The commutativity and the smoothness of $(b,\d)$ are easily checked. The point $z\in Z^{B}_{\s}(k)$ is constructed similarly as in the proof of  Lemma \ref{l:ov W}, with the full flags $\{\cF(\t)\}_{\t\in \Sig}$ (one for each fiber $\cE|_{\t(s)}$, $\t\in \Sig$) in the original point $(s,\cE, \{\cF(\t)\}_{\t\in \Sig})\in\Bun_{G,B}(\pi,\s)(k)$ as the extra data needed to lift a point from $Z$ to $Z^{B}_{\s}$. 
\end{proof}

\begin{remark} It is possible to define a version of the universal Eisenstein cone for the moduli stack of $G$-bundles with arbitrarily deep level structure $\bK$ along $\s$, in the same style as we did for the Iwahori level. The idea is to treat principal congruent level  $\Bun_{G,n}(\pi,\s)$ first (in which case the cotangent bundle $T^{*}\Bun_{G,n}(\pi,\s)$ descends to $\Bun_{G}(\pi)$, just as in the case $n=1$ considered above), and then transport the Eisenstein cone there using $\orr{\frp_{n}}$, where $\frp_{n}:\Bun_{G,n}(\pi,\s)\to \Bun_{G,\bK}(\pi, \s)$ is defined for large enough $n$. 
\end{remark}

\section{Nilpotent sheaves under Hecke operators}

The goal of this section is to show that the category of sheaves with singular support in the universal Eisenstein cone are stable under Hecke functors. This is a generalization of our earlier result \cite[Theorem 5.2.1]{NY}.

In this section we will work over the base field $k=\CC$. Fix a field $E$ of characteristic zero as the field of coefficients for sheaves. For a stack $\frM$ over $\CC$, we denote by $\Sh(\frM)$ the cocomplete category of complexes of sheaves on $\frM$ with $E$-coefficients under the classical topology.

Consider a family of curves $\pi:\frX\to S$ as in \S\ref{ss:rel curve}, with $S$ smooth.

\subsection{Nilpotent sheaves}\label{ss:nilp sh} We will write
\begin{equation*}
\cN_{\pi}:=\cN^{\Eis}_{\pi}.
\end{equation*}
Consider 
\begin{equation*}
\Sh_{\cN^{\Eis}_{\pi}}(\Bun_{G}(\pi))\subset \Sh(\Bun_{G}(\pi)),
\end{equation*}
the full subcategory consisting of sheaves with singular support contained in $\cN_{\pi}$.

For any base change $s: S'\to S$, where both $S'$ and $S$ are smooth over $k$, recall that we have an induced map $\th: \Bun_{G}(\pi')\to \Bun_{G}(\pi)$.

\begin{lemma}\label{l:nilp pullback}
In the above situation,  $\th^{*}$ sends $\Sh_{\cN_{\pi}}(\Bun_{G}(\pi))$ to $\Sh_{\cN_{\pi'}}(\Bun_{G}(\pi'))$.
\end{lemma}
\begin{proof}
We claim that $\th$ is non-characteristic with respect to $\cN_{\pi}$ in the sense of \cite[Definition 5.4.12]{KS}. Indeed, in the Lagrangian correspondence
\begin{equation*}
\xymatrix{T^{*}\Bun_{G}(\pi') & S'\times_{S}T^{*}\Bun_{G}(\pi) \ar[l]_-{d\th}\ar[r]^-{\th^{\na}} & T^{*}\Bun_{G}(\pi)}
\end{equation*}
We have $(d\th)^{-1}(0_{\Bun_{G}(\pi'})$ is contained in the pullback of $S'\times_{S}d\Pi(T^{*}S)$. By Corollary \ref{c:univ cone non-char}, $\cN_{\pi}\cap d\Pi(T^{*}S)$ is the zero section, hence $(\th^{\na})^{-1}(\cN_{\pi})=S'\times_{S}\cN_{\pi}$ also intersects $(d\th)^{-1}(0_{\Bun_{G}(\pi')})$ in the zero section. Therefore $\th$ is non-characteristic with respect to $\cN_{\pi}$, i.e., $(\th^{\na})^{-1}(\cN_{\pi})\cap(d\th)^{-1}(0_{\Bun_{G}(\pi')})$ is contained in the zero section.

By \cite[Proposition 5.4.13]{KS}, for $\cF\in \Sh_{\cN_{\pi}}(\Bun_{G}(\pi))$, since $\th$ is non-characteristic with respect to $\cN_{\pi}$, $\ssupp(\th^{*}\cF)$ is contained in $\oll{\th}(\ssupp(\cF))\subset \oll{\th}(\cN_{\pi})$, which is $\cN_{\pi'}$ by Lemma \ref{l:Eis cone bc}.
\end{proof}

\subsection{Local constancy of automorphic category}
As mentioned in the Introduction, Betti geometric Langlands predicts that $\Sh_{\cN_{s}}(\Bun_{G}(\frX_{s}))$ moves locally constantly as $s$ moves in $S$. A natural approach to establish local constancy is to construct a ``connection'' on the sheaf of categories $\Sh_{\cN_{s}}(\Bun_{G}(\frX_{s}))$ by moving $s\in S$.  Imposing a global singular support condition for sheaves on $\Sh_{\cN_{s}}(\Bun_{G}(\frX_{s}))$ serves the purpose of such a connection.
 
\begin{remark}\label{r:analytic setup} To formulate the local constancy conjecture, we will need to consider more general bases $S$ then algebraic stacks. We still start with the an algebraic setup as in \S\ref{ss:rel curve}, and assume $S$ is a smooth DM stack over $\CC$. Now $U$ be a complex manifold and $U\to S^{an}$ be a complex analytic map that is also \'etale (locally biholomorphic). Denote $\pi_{U}: U\times_{S}\frX\to U$ the base change of $\pi$, an analytic family of curves. It makes sense to form the base change $\Bun_{G}(\pi_{U})=U\times_{S}\Bun_{G}(\pi)$ as a complex analytic stack. We turn the base change property of the universal Eisenstein cone proved in Lemma \ref{l:Eis cone bc} into a definition in the complex analytic setting: define $\cN_{\pi_{U}}\subset T^{*}\Bun_{G}(\pi_{U})$ to be the image of $U\times_{S}\cN_{\pi}$ under the map $d\th: U\times_{S}T^{*}\Bun_{G}(\pi)\to T^{*}\Bun_{G}(\pi_{U})$.  

It is possible to make base change to a general complex manifold $S'$, or even a real manifold, using the formalism of relative real analytic spaces developed in \cite[Appendix A.2-A.3]{NY-glue}, but we do not need such generality here.
\end{remark} 

Let $U\to S^{an}$ be an analytic \'etale map from a complex manifold $U$ and let $u\in U$. Let $i_{u}: \Bun_{G}(\frX_{u})\incl \Bun_{G}(\pi_{U})$ be the inclusion of a fiber. Then Lemma \ref{l:nilp pullback} gives a functor
\begin{equation}\label{res Ms}
i_{u}^{*}: \Sh_{\cN_{\pi_{U}}}(\Bun_{G}(\pi_{U}))\to \Sh_{\cN_{u}}(\Bun_{G}(\frX_{u}))
\end{equation}
where on the right side $\cN_{u}$ denotes the usual global nilpotent cone inside $T^{*}\Bun_{G}(\frX_{u})$. 

\begin{conj}\label{c:res eq} In the above setting, suppose $U$ is  contractible, then the restriction functor \eqref{res Ms} is an equivalence.
\end{conj}

\begin{remark} Recently Faergeman and Kjaersgaard obtained a proof of Conjecture \ref{c:res eq}. As far as we understand, the proof is a direct topological argument in the case of rank one groups, but for general $G$ it uses deep ingredients from the proof of the geometric Langlands conjecture. A direct topological argument in general still seems to be missing.
\end{remark}

\sss{Mapping class group action} Let $s\in S(\CC)$ and let $S'\to S^{an}$ be the universal cover of $S^{an}$ based at $s$: it is equipped with a base point $s'\in S'$ lifting $s$. Then $\pi_{1}(S^{an},s)$ acts on $S'$ by deck transformations. It induces an action of $\pi_{1}(S^{an},s)$ on $\Bun_{G}(\pi_{S'})\cong S'\times_{S}\Bun_{G}(\pi)$ preserving the universal Eisenstein cone $\cN_{\pi_{S'}}$. The category $\Sh_{\cN_{\pi_{S'}}}(\Bun_{G}(\pi_{S'}))$ therefore carries an action of $\pi_{1}(S^{an},s)$. 

Assume that Conjecture \ref{c:res eq} holds. Assume also that $S'$ is contractible, i.e., $S^{an}$ is a $K(\pi,1)$ space. Conjecture \ref{c:res eq} then gives an equivalence $\Sh_{\cN_{\pi_{S'}}}(\Bun_{G}(\pi_{S'}))\isom \Sh_{\cN_{s'}}(\Bun_{G}(\frX_{s'}))$, the latter is the same as $\Sh_{\cN_{s}}(\Bun_{G}(\frX_{s}))$. We conclude that there is a canonical action of $\pi_{1}(S^{an},s)$ on $\Sh_{\cN_{s}}(\Bun_{G}(\frX_{s}))$ in this case.

Now starting with a connected smooth projective curve $X$ over $\CC$ of genus $g\ge2$. Consider $X$ as the fiber of the universal family of genus $g$ curves $\pi_{g}: \frX_{g}\to \frM_{g}$ over $\xi\in \frM_{g}(\CC)$. Applying the above reasoning, and using the fact that the Teichm\"uller space (the universal cover of $\frM_{g}^{an}$ as an orbifold) is contractible, we obtain a canonical action of the mapping class group $\pi_{1}(\frM^{an}_{g}, \xi)$ on $\Sh_{\cN_{X}}(\Bun_{G}(X))$.

\subsection{Spherical Hecke functor}\label{ss:Hk} 
Now suppose $S'$ is another stack smooth over $k$ with a map $t: S'\to S$ that lifts to $\frX$ 
\begin{equation*}
\xymatrix{ & \frX\ar[d]^{\pi}\\
S'\ar[r]^{t}\ar[ur]^{\xi} & S}
\end{equation*}
We use the notation $\pi':S'\times_{S}\frX\to S'$ etc. from \S\ref{sss:bc}. 

Let $G\lr{z}$ be the loop group of $G$ and $G\tl{z}$ be the arc group. Form the local Hecke stack $G\tl{z}\bs G\lr{z}/G\tl{z}$, with the action of the pro-algebraic group $\Aut(k\tl{z})$ of automorphisms of the formal disk $\Spf k\tl{z}$. For a $k$-algebra $R$, an $R$-point of the quotient stack $\frac{G\tl{z}\bs G\lr{z}/G\tl{z}}{\Aut(k\tl{z})}$ classifies $(\wh D, \cE_{1},\cE_{2}, \t)$ where $\wh D$ is a formal disk over $R$ (i.e., \'etale locally on $\Spec R$ isomorphic to $\Spf R\tl{z}$, and to $\wh D$ we can canonically assign an $R$-scheme $D$ locally isomorphic to $\Spec R\tl{z}$), $\cE_{1}$ and $\cE_{2}$ two $G$-bundles over $D$ and $\t$ is an isomorphism between the pullbacks $\cE_{1}|_{D^{\times}}$ and $\cE_{2}|_{D^{\times}}$ (where $D^{\times}=D\bs \Spec R$).

Consider the Hecke correspondence 
\begin{equation*}
\xymatrix{  & \Hk^{\Sph}_{\xi} \ar[rr]^-{\inv^{\Sph}_{\xi}}\ar[dr]^{p_{2}}\ar[dl]_{p_{1}}  && \frac{G\tl{z}\bs G\lr{z}/G\tl{z}}{\Aut(k\tl{z})}\\
\Bun_{G}(\pi) &  & \Bun_{G}(\pi')  
}
\end{equation*}
whose $S''$-points, for any $k$-scheme $S''$, classifies $(s', \cE_{1}, \cE_{2}, \t)$ where $s': S''\to S'$, $\cE_{1},\cE_{2}$ are $G$-bundles over $\frX_{S''}:=S''\times_{S}\frX$ (using the map $t\c s': S''\to S$) and  $\t$ is an isomorphism between  $\cE_{1}$ and $\cE_{2}$ over $\frX_{S''}\bs \G(\xi\c s')$, the complement of the graph of $\xi\c s': S''\to \frX$. The map $p_{1}$ records $(t\c s',\cE_{1})$ and $p_{2}$ records $(s', \cE_{2})$. The map $\inv^{\Sph}_{\xi}$ records the formal completion of $\frX_{S''}$ along the graph $\G(\xi\c s')$ together with the restrictions of $(\cE_{1},\cE_{2}, \t)$ to that formal disk over $S''$.

Consider the spherical Hecke category 
\begin{equation*}
\cH^{\Sph}=\Sh\left(\frac{G\tl{z}\bs G\lr{z}/G\tl{z}}{\Aut(\CC\tl{z})}\right)
\end{equation*}
This can be identified with the equivariant category of sheaves on the affine Grassmannian $\Gr_{G}=G\lr{z}/G\tl{z}$ for the action of $G\tl{z}\rtimes\Aut(\CC\tl{z})$ by left translation and change of the local coordinate $z$. For $\cK\in \cH^{\Sph}$, we have the Hecke functor
\begin{equation*}
H^{\Sph}_{\xi,\cK}=p_{2!}(p^{*}_{1}(-)\ot \inv^{\Sph,*}_{\xi}\cK): \Sh(\Bun_{G}(\pi))\to \Sh(\Bun_{G}(\pi'))
\end{equation*}

\begin{theorem}\label{th:sph Hk pres} For any $\cK\in \cH^{\Sph}$, the Hecke functor $H^{\Sph}_{\xi,\cK}$ sends $\Sh_{\cN_{\pi}}(\Bun_{G}(\pi))$ to $\Sh_{\cN_{\pi'}}(\Bun_{G}(\pi'))$. 
\end{theorem}

\begin{remark} We explain how Theorem \ref{th:sph Hk pres} recovers \cite[Theorem 5.2.1]{NY}. The latter says that for a single curve $X$ and Hecke modification $H_{\cK}: \Sh(\Bun_{G}(X))\to \Sh(X\times \Bun_{G}(X))$ for a moving point in $X$, then  $H^{\Sph}_{\cK}$ sends $\Sh_{\cN_{X}}(\Bun_{G}(X))$ to $\Sh_{0_{X}\times \cN_{X}}(X\times \Bun_{G}(X))$ (i.e., singular support of the result is contained in the product of the zero section of $T^{*}X$ and the global nilpotent cone in $T^{*}\Bun_{G}(X)$).

To recover this statement, we only need to apply Theorem \ref{th:sph Hk pres} to the constant curve $\frX=X$ over $S=\pt$ and $S'=X, \xi=\id_{X}:S'=X\to X$. Note that $\Bun_{G}(\pi')=X\times \Bun_{G}(X)$ and its universal Eisenstein cone is $0_{X}\times \cN_{X}$ by the base change property in Lemma \ref{l:Eis cone bc}.
\end{remark}

\begin{remark} If we take $\cK$ to be the monoidal unit in $\cH^{\Sph}$, $H_{\xi,\cK}^{\Sph}$ is the same as $\th^{*}$ considered in Lemma \ref{l:nilp pullback}, of which this theorem is a generalization. 
\end{remark}

We prove this theorem by first reducing to the Iwahori version (in \S\ref{sss:red to Iw}) and then prove the Iwahori version in \S\ref{sss:proof Iw}.

\subsection{Affine Hecke category with universal monodromy}\label{sss:aff Hk}
We shall state a version of Theorem \ref{th:sph Hk pres} with $B$ or $N$-level structures. Let $\bI$ (resp. $\bI^{\circ}$) be the preimage of $B\subset G$ (resp. $N\subset G$) under the reduction mod $z$ map $G\tl{z}\to G$. The (universally monodromic) affine Hecke category is the monoidal category
\begin{equation*}
\cH=\Sh_{\bimon}(\bI^{\circ}\bs G\lr{z}/\bI^{\circ})
\end{equation*}
under convolution. Here the subscript $\bimon$ means we are taking only sheaves that are monodromic under both the left and the right translation under $H=B/N$ on $\bI^{\circ}\bs G\lr{z}/\bI^{\circ}$. The monoidal structure is defined by the usual convolution diagram
\begin{equation*}
\xymatrix{ & \bI^{\circ}\bs G\lr{z}\times^{\bI^{\circ}} G\lr{z}/\bI^{\circ} \ar[dl]_{p_{1}}\ar[dr]^{p_{2}}\ar[rr]^-{m}&& \bI^{\circ}\bs G\lr{z}/\bI^{\circ}\\
\bI^{\circ}\bs G\lr{z}/\bI^{\circ} && \bI^{\circ}\bs G\lr{z}/\bI^{\circ}}
\end{equation*}
with the formula
\begin{equation*}
\cK_{1}\star\cK_{2}:=m_{!}(p_{1}^{*}\cK_{1}\otimes p_{2}^{*}\cK_{2})[2\dim H].
\end{equation*}
Let $\cL_{\univ}$ be the universal local system on $H$ whose stalks are canonically isomorphic to the group ring of $E[\pi_1(H)]$ as the regular representation of $\pi_1(H)$. Let $\d: H\simeq \bI/\bI^{\c}\incl G\lr{z}/\bI^{\circ}$ be the natural map. Then $\d_{*}\cL_{\univ}\in \cH$ is an identity object for the monoidal structure on $\cH$. 

The monoidal category $\Sh_{0}(H)$ of local systems on $H$ can be identified with a full monoidal subcategory of $\cH$ of objects supported on $\bI$. Left and right convolution with $\Sh_{0}(H)$ gives $\cH$ the structure of a $\Sh_{0}(H)$-bimodule.

Let $\tilW=\xcoch(T)\rtimes W$ be the extended affine Weyl group. For each $w\in \tilW$, fix a lifting $\dot w\in N_{G\lr{z}}(T)$. Let $i_{w}: \bI^{\c}\bs \bI \dot{w} \bI/\bI^{\c}\incl \bI^{\c}\bs G\lr{z}/\bI^{\circ}$ be the inclusion. The map $h_{\dot w}: H\to\bI^{\c}\bs \bI \dot{w} \bI/\bI^{\c}$ mapping $x\mapsto x\dot w$ is a gerbe for a pro-unipotent group. Therefore the universal local system $\cL_{\univ}$ descends to a local system $\cL_{\univ}(w)$ on $\bI^{\c}\bs \bI \dot{w} \bI/\bI^{\c}$. Let $\wh\D(w):=i_{w!}\cL_{\univ}( w)[\ell(w)]$. We have $\d\cong \wh\D(1)$. The objects $\{\wh\D(w)\}_{w\in W}$ and their shifts generate $\cH$ under colimits.

\subsection{Iwahori version}\label{sss:aff Hk action}
In the setup of \S\ref{ss:rel curve}, assume $\s\subset \frX$ contains the image of a section $\xi:S\to \frX$ to $\pi$. Write $\s=\xi(S)\coprod \s'$. We simply write
\begin{equation*}
\Sh_{\cN}(\Bun_{G,N}(\pi,\s)):=\Sh_{\cN^{N}_{\pi,\s}}(\Bun_{G,N}(\pi,\s)).
\end{equation*}
Suppose we are given a trivialization of the formal neighborhood of $\frX$ along the image of $\xi$:
\begin{equation*}
\frX^{\wedge}_{/\xi(S)}\isom \Spf(k\tl{z})\wh\times S. 
\end{equation*}

Let $\Hk_{\xi}$ be the analog of $\Hk_{\xi}^\sph$ with $N$-level structures along $\s$. More precisely, for any scheme $S'$ over $\CC$, $\Hk_{\xi}(S')$ is the groupoid of tuples $(s,\cE_1,\cE_{1,\s,N},\cE_2,\cE_{2,\s,N},\t)$ where $s:S'\to S$, $(s,\cE_i, \cE_{i,\s,N})\in \Bun_{G,N}(\pi,\s)(S')$ for $i=1,2$, and $\t$ is an isomorphism between $(\cE_1,\cE_{1,\s',N})$ and $(\cE_2,\cE_{2,\s',N})$ over $\frX_s-\xi_s=(\frX-\xi)\times_{S}S'$ (in particular, $\t$ preserves the $N$-reductions along $\s'\times_{S}S'$). Consider the diagram
\begin{equation*}
\xymatrix{  & \Hk_{\xi} \ar[rr]^-{\inv_{\xi}}\ar[dr]^{q_{2}}\ar[dl]_{q_{1}}  && \bI^{\c}\bs G\lr{z}/\bI^{\c}\\
\Bun_{G,N}(\pi,\s) &  & \Bun_{G,N}(\pi,\s)  
}
\end{equation*}
where $q_i$ records $(s,\cE_i, \cE_{i,\s,N})$ and $\inv_\xi$ records the relative position between $(s,\cE_1,\cE_{1,\xi,N})$ and $(s,\cE_2,\cE_{2,\xi,N})$ along $\xi_s$.

For $\cK\in \cH$, its action on $\Sh(\Bun_{G,N}(\pi,\s))$ is given by
\begin{equation*}
H_{\xi,\cK}=q_{2!}(q_{1}^{*}(-)\ot \inv^{*}_{\xi}\cK): \Sh(\Bun_{G,N}(\pi,\s))\to \Sh(\Bun_{G,N}(\pi,\s)).
\end{equation*}

\begin{theorem}\label{th:aff Hk pres} In the above situation, for any $\cK\in \cH$, the functor $H_{\xi,\cK}$ preserves $\Sh_{\cN}(\Bun_{G,N}(\pi,\s))$.
\end{theorem}

\sss{Theorem \ref{th:aff Hk pres} implies Theorem \ref{th:sph Hk pres}}\label{sss:red to Iw} 

Let $\th: \Bun_{G}(\pi')\to\Bun_{G}(\pi)$ be the map induced from $t:S'\to S$. We may factor $p_{1}$ as the composition
\begin{equation*}
p_{1}: \Hk_{\xi}^{\Sph}\xr{p'_{1}} \Bun_{G}(\pi')\xr{\th} \Bun_{G}(\pi)
\end{equation*}
By Lemma \ref{l:nilp pullback}, for $\cF\in \Sh_{\cN}(\Bun_{G}(\pi))$, $\th^{*}\cF\in \Sh_{\cN}(\Bun_{G}(\pi'))$. If we work with the family $\pi': \frX'\to S'$ together with the section $\xi'=(\xi,\th):S'\to \frX'$, and the associated Hecke functor $H^{\Sph}_{\xi',\cK}$ for this family, we have $H^{\Sph}_{\xi,\cK}(\cF)=H^{\Sph}_{\xi',\cK}(\th^{*}(\cF))$. Therefore we may reduce the general case to the case $S'=S$ and $\xi$ is a section of $\pi$. Zariski locally over $S$ we may trivialize the formal neighborhood of $\xi$ in $\frX$. We are in the same setting as the Iwahori version \S\ref{sss:aff Hk action}, with $\s=\xi(S)$.

Let $r': \Bun_{G,N}(\pi,\s)\to \Bun_{G}(\pi)$ be the projection. For $\cK\in\cH^{\Sph}$, let $\wt\cK\in \cH$ be the pullback via the projection $\bI^{\c}\bs G\lr{z}/\bI^{\c}\to (G\tl{z}\bs G\lr{z}/G\tl{z})/\Aut(k\tl{z})$. Then for $\cF\in \Sh(\Bun_{G}(\pi))$ 
we have
\begin{equation*}
H_{\xi,\wt\cK}(r'^{*}\cF) \simeq r'^{*}H^{\Sph}_{\xi, \cK}(\cF) \otimes \cohoc{*}{G/N}
\end{equation*}
By Theorem \ref{th:aff Hk pres}, $\ssupp(H_{\xi,\wt\cK}(r'^{*}\cF))\subset \cN^{\Eis,N}_{\pi,\s}$, hence $\oll{r'}\ssupp(H^{\Sph}_{\xi, \cK}(\cF))=\ssupp(r'^{*}H^{\Sph}_{\xi, \cK}(\cF))\subset \cN^{\Eis,N}_{\pi,\s}$. Since $r'$ is smooth and surjective, $\ssupp(H^{\Sph}_{\xi, \cK}(\cF))=\orr{r'}\oll{r'}\ssupp(H^{\Sph}_{\xi, \cK}(\cF))\subset \orr{r'}\cN^{\Eis,N}_{\pi,\s}=\cN^{\Eis}_{\pi}$, the last equality being contained in Remark \ref{rem: cone B vs N}.  We conclude that $H^{\Sph}_{\xi, \cK}$ preserves $\Sh_{\cN_\pi}(\Bun_G(\s))$.

%
%

\sss{Proof of Theorem \ref{th:aff Hk pres}}\label{sss:proof Iw} To simplify notation, we assume $\s=\xi(S)$, i.e., $\s'=\vn$. Also we will write $\Bun_{G,N}$ for $\Bun_{G,N}(\pi,\s)$, etc.

We first reduce the proof to a calculation of transport of Lagrangians. Let $S_\aff$ be the set of simple reflection in the affine Weyl group $W_\aff$, and let $\Om\subset \tilW$ be the subgroup of length zero elements. Then the objects $\{\wh\D(s)|s\in S_{\aff}\}$ and $\{\wh\D(\om)|\om \in \Om\}$ generate $\cH$ monoidally. Therefore it suffices to check $H_{\xi, \wh\D(s)}$ and  $H_{\xi,\wh\D(\om)}$ preserve nilpotent sheaves.


The action of $H_{\xi,\wh\D(s)}$ ($s$ is a simple reflection in $\tilW$) can be expressed as follows. Let $\bP_{s}\subset G\lr{z}$ be the standard parahoric corresponding to $s$. Consider the moduli stack $\Bun_{G,\bP_{s}}$ classifying $G$-bundles along the fibers of $\pi$ together with $\bP_{s}$-level structures along $\s=\xi(S)$. Let $\Hk_{\le s}:=\Bun_{G,N}\times_{\Bun_{G,\bP_{s}}}\Bun_{G,N}$ be the preimage of $\bI^{\c}\bs \bP_{s}/\bI^{\c}$ in $\Hk_{\xi}$ under $\inv_{\xi}$. Below we will treat the slightly more general  situation where $\wh\D(s)$ is replaced by any object $\cK\in \cH$ supported on $\bI^{\c}\bs \bP_{s}/\bI^{\c}$.


Note that $H$ acts on $\Bun_{G,N}$ by changing the $N$-reductions along $\s$. Let $H_{>0}$ be the identity component of $H(\RR)$. Let $\Bun_{G,NH_{>0}}=\Bun_{G,N}(\pi,\s)/H_{>0}$ as a real analytic stack. We have a universal Eisenstein cone $\cN^{\Eis, NH_{>0}}_{\pi,\s}\subset T^*\Bun_{G,NH_{>0}}$, defined as the transport of $\cN^{\Eis, B}_{\pi,\s}$ along the projection $\Bun_{G,NH_{>0}}\to \Bun_{G,B}$, or equivalently as the transport of $\cN^{\Eis, N}_{\pi,\s}$ along the projection $\Bun_{G,N}\to \Bun_{G,NH_{>0}}$.

The Hecke stack $\Hk_\xi$ has an action by $H\times H$ by changing the $N$-reductions $\cE_{1,\xi,N}$ and $\cE_{2,\xi,N}$. Let $\ov\Hk=\Hk/(H_{>0}\times H_{>0})$. It fits into a diagram
\begin{equation}\label{ov Hk}
\xymatrix{& \ov\Hk\ar[rr]^-{\inv_\xi}\ar[dr]^{ q_{2}}\ar[dl]_{  q_{1}}  && \bI^{\c}H_{>0}\bs G\lr{z}/\bI^{\c}H_{>0}\\
\Bun_{G,NH_{>0}} &  & \Bun_{G,NH_{>0}}  
}
\end{equation}

For $w\in \tilW$, let $\ov\Hk_{w}$ be the preimage of $\bI^{\c}H_{>0}\bs \bI w \bI/\bI^{\c}H_{>0}$ under $\inv_\xi$. Let $q_{i,w}: \ov\Hk_{w}\to \Bun_{G,NH_{>0}}$ be the restriction of $q_i$ to $\ov\Hk_w$, for $i=1,2$.

For a simple reflection $s$, let $\ov\Hk_{\le s}$ be the preimage of $\bI^{\c}H_{>0}\bs \bP_s/\bI^{\c}H_{>0}$ under $\inv_\xi$. Then $\ov\Hk_{\le s}\cong \Bun_{G,NH_{>0}}\times_{\Bun_{G,\bP_{s}}}\Bun_{G,NH_{>0}}$. Restricting \eqref{ov Hk} to $\ov\Hk_{\le s}$ we get a diagram
\begin{equation*}
\xymatrix{& \ov\Hk_{\le s}\ar[rr]^-{\inv_{\le s}}\ar[dr]^{ q_{2,\le s}}\ar[dl]_{  q_{1,\le s}}  && \bI^{\c}H_{>0}\bs \bP_{s}/\bI^{\c}H_{>0}\\
\Bun_{G,NH_{>0}} &  & \Bun_{G,NH_{>0}}  
}
\end{equation*}
Note that $ q_{1,\le s}$ and $q_{2,\le s}$ are now proper with fibers $\bP_{s}/\bI^{\c}H_{>0}$ (a fiber bundle over $\PP^{1}$ with fibers the compact torus $H(\CC)/H_{>0}$). We have $ \ov\Hk_{\le s}=\ov \Hk_{s}\cup \ov \Hk_{1}$ according to the decomposition $\bP_{s}=\bI s\bI\cup \bI$. Let $i: \ov \Hk_{1}\incl \ov \Hk_{\le s}$ and $j: \ov \Hk_{s}\incl \ov \Hk_{\le s}$ be the closed and open embeddings. 

All objects $\cK\in \cH$ descend to $\bI^{\c}H_{>0}\bs G\lr{z}/\bI^{\c}H_{>0}$ by monodromicity. We denote the descent still by $\cK$. Also, any object $\cF\in \Sh_{\cN}(\Bun_{G,N})$ descends to $\Bun_{G,NH_{>0}}$ by $H$-monodromicity, which we still denote by $\cF$.

The map $(q_{1,\le s}, \inv_{\le s}): \ov\Hk_{\le s}\to \Bun_{G,NH_{>0}}\times \bI^{\c}H_{>0}\bs \bP_{s}/\bI^{\c}H_{>0}$ is in fact smooth (submersive). Therefore $\ssupp(q_{1,\le s}^{*}\cF\ot\inv_{\le s}^{*}\cK)$ is the transport of $\ssupp(\cF)\times\ssupp(\cK)$ under $\oll{(q_{1,\le s}, \inv_{\le s})}$. Now $\ssupp(\cK)$ is contained in the union of the zero section and the conormal bundle of the closed stratum $\bI^{\c}\bs \bI/\bI^{\c}$. We conclude that $\ssupp(q_{1,\le s}^{*}\cF\ot\inv_{\le s}^{*}\cK)$ is contained in the union 
\begin{equation*}
\oll{q_{1,\le s}}(\cN^{\Eis,NH_{>0}}_{\pi,\s})\cup \orr{i}\oll{q_{1,1}}(\cN^{\Eis,NH_{>0}}_{\pi,\s}).
\end{equation*}
Since $q_{2,\le s}$ is proper, the singular support of $H_{\xi,\cK}(\cF)=q_{2,\le s *}(q_{1,\le s}^{*}\cF\ot\inv_{\le s}^{*}\cK)$ is contained in
\begin{equation}\label{estimate ss}
\orr{q_{2,\le s}}\oll{q}_{1,\le s}(\cN^{\Eis,NH_{>0}}_{\pi,\s})\cup \orr{q_{2,\le s}}\orr{i}\oll{q_{1,1}}(\cN^{\Eis,NH_{>0}}_{\pi,\s})\subset T^{*}\Bun_{G,NH_{>0}}.
\end{equation}
Note $q_{2,1}=q_{2,\le s}\c i$, hence $\orr{q_{2,\le s}}\orr{i}=\orr{q_{2,1}}$.

For an isotropic subset $\L\subset T^*\Bun_{G,B}$, we denote by $\L_{top}$ the union of its irreducible components of dimension $\dim \Bun_{G,B}$. In other words, $\L_{top}$ is the maximal Lagrangian in $\L$. Since the singular support of $H_{\xi,\cK}(\cF)$ is always a Lagrangian, in view of \eqref{estimate ss}, to show that $H_{\xi,\cK}(\cF)\subset \cN^{\Eis,NH_{>0}}_{\pi,\s}$,  it suffices to show
\begin{eqnarray}\label{Hk s Lag}
\left(\orr{q_{2,\le s}}\oll{q_{1,\le s}}(\cN^{\Eis,NH_{>0}}_{\pi,\s})\right)_{top}\subset \cN^{\Eis,NH_{>0}}_{\pi,\s}\\
\label{Hk 1 Lag}\left(\orr{q_{2,1}}\oll{q_{1,1}}(\cN^{\Eis,NH_{>0}}_{\pi,\s})\right)_{top}\subset \cN^{\Eis,NH_{>0}}_{\pi,\s}.
\end{eqnarray}

Similarly, for a length zero element $\om\in \Om$, restricting the diagram \eqref{ov Hk} to $\ov\Hk_{\om}$ gives a self-correspondence of $\Bun_{G,NH_{>0}}$
\begin{equation*}
    \xymatrix{& \ov \Hk_\om\ar[dl]_{q_{1,\om}}\ar[dr]^{ q_{2,\om}}\\
    \Bun_{G,NH_{>0}} & & \Bun_{G,NH_{>0}}}
\end{equation*}
Now $q_{i,\om}$ are torsors for the compact torus $H(\CC)/H_{>0}$. The same argument as in the case of a simple reflection shows that, $H_{\xi,\wh\D(\om)}$ preserves nilpotent sheaves if we can show
\begin{equation}\label{Hk om Lag}
\left(\orr{q_{2,\om}}\oll{q_{1,\om}}(\cN^{\Eis,NH_{>0}}_{\pi,\s})\right)_{top}\subset \cN^{\Eis,NH_{>0}}_{\pi,\s}.
\end{equation}

Now we show \eqref{Hk s Lag}, and the arguments for \eqref{Hk 1 Lag} and \eqref{Hk om Lag} are simpler and will be omitted.  Since $\cN^{\Eis,NH_{>0}}_{\pi,\s}$ is the transport of $\cN^{\Eis,B}_{\pi,\s}$ from $T^{*}\Bun_{G,B}$, it suffices to prove the analogue of \eqref{Hk s Lag} for $\Bun_{G,B}$. Namely in the diagram
\begin{equation*}
\xymatrix{\Bun_{G,B} & \Hk_{\le s}=\Bun_{G,B}\times_{\Bun_{G,\bP_{s}}}\Bun_{G,B}\ar[l]_-{r_{1}}\ar[r]^-{r_{2}} & \Bun_{G,B}}
\end{equation*}
we need to show
\begin{equation}\label{Hk s NB}
\left(\orr{r_{2}}\oll{r_{1}}(\cN^{\Eis,B}_{\pi,\s})\right)_{top}\subset \cN^{\Eis,B}_{\pi,\s}.
\end{equation}
The notation $(-)_{top}$ makes sense since by Lemma \ref{l:pres iso}, $\orr{r_{2}}\oll{r_{1}}(\cN^{\Eis,B}_{\pi,\s})$ is a conic isotropic constructible subset of $T^{*}\Bun_{G,B}$, so
\begin{equation*}
\L':=\left(\orr{r_{2}}\oll{r_{1}}(\cN^{\Eis,B}_{\pi,\s})\right)_{top}
\end{equation*}
is the maximal Lagrangian (automatically conic) inside $\orr{r_{2}}\oll{r_{1}}(\cN^{\Eis,B}_{\pi,\s})$. It is easy to see that $\orr{r_{2}}$ and $\oll{r_{1}}$ preserve the non-characteristic property with respect to the projections to $S$, therefore $\L'$ is non-characteristic with respect to $\Pi^{B}_{\s}: \Bun_{G,B}\to S$. By Proposition \ref{p:contained in lifting}, to show that $\L'$ is contained in $\cN^{\Eis, B}_{\pi,\s}$, which is the lifting of $\cN^{\rel,B}_{\pi,\s}$ by Theorem \ref{th:univ cone Iw}, it suffices to show that the image of $\L'$ in the relative cotangent $T^{*}(\Pi^{B}_{\s})$ is contained in $\cN^{\rel,B}_{\pi,\s}$. Let $\oll{r_{1}}^{\rel}$ and $\orr{r_{2}}^{\rel}$ be the relative transport (with respect to $S$) defined as in \S\ref{sss:rel transport}, then to show \eqref{Hk s NB}, it boils down to showing
\begin{equation*}
\orr{r_{2}}^{\rel}\oll{r_{1}}^{\rel}(\cN^{\rel,B}_{\pi,\s})\subset \cN^{\rel,B}_{\pi,\s}.
\end{equation*}
This can be checked by restricting to each closed point $s\in S$, for which the statement is clear, since Hecke modification on Higgs bundles does not change the Higgs bundle at the generic point of the curve $\frX_{s}$, so in particular nilpotence is preserved. This finishes the proof of Theorem \ref{th:aff Hk pres}.

\qed

\end{document}